\documentclass[11pt]{amsart}
\usepackage{amssymb,amsmath,amsthm}
\usepackage{a4wide}
\usepackage{graphicx}
\usepackage{amsfonts}
\usepackage{amssymb}

\usepackage{color}
\usepackage{mathrsfs,a4wide}
\usepackage{latexsym}
\usepackage{bbm}

\numberwithin{equation}{section}

\theoremstyle{plain}
\begingroup
\newtheorem{theorem}{Theorem}[section]
\newtheorem{lemma}[theorem]{Lemma}
\newtheorem{proposition}[theorem]{Proposition}

\endgroup

\theoremstyle{definition}
\begingroup
\newtheorem{definition}[theorem]{Definition}
\newtheorem{remark}[theorem]{Remark}

\endgroup

\newcommand{\Reg}{\mathfrak{C}}

\def\Ins{{\mathfrak{M}}}
\newcommand{\F}{\mathcal{F}}

\newcommand{\N}{\mathbb{N}}
\newcommand{\Z}{\mathbb{Z}}
\newcommand{\R}{\mathbb{R}}

\newcommand{\loc}{\mathrm{loc}}
\newcommand{\Q}{\mathscr{Q}}

\newcommand{\ep}{\varepsilon}

\newcommand{\ud}{\,\mathrm{d}}
\def\S{\mathbb{S}}
\def\T{\mathbb{T}}

\def\F{\mathcal{F}}
\def\J{\mathscr{J}}

\def\D{\mathscr{D}}

\newcommand{\newatop}{\genfrac{}{}{0pt}{1}} 

\newcommand{\per}{\mathscr{P}}

\newcommand{\dist}{\mathrm{dist}}
\newcommand{\osc}{\mathrm{osc}}
\newcommand{\I}{\mathscr{G}}
\newcommand{\sdist}{\mathrm{sdist}}

\newcommand{\Eh}{E^{(h)}}

\newcommand{\fat}{\mathrm{fat}}

\definecolor{ddorange}{rgb}{1,0.5,0}
\definecolor{ddcyan}{rgb}{0,0.2,1.0}

\title[Flat flows of periodic Lipschitz subgraphs for generalized nonlocal perimeters] {Flat flows of periodic Lipschitz subgraphs for generalized nonlocal perimeters}

\author[L. De Luca]
{Lucia De Luca}
\address[Lucia De Luca]{Istituto per le Applicazioni del Calcolo ``M. Picone'' IAC-CNR, Via dei Taurini 19, I-00185 Roma, Italy}
\email[L. De Luca]{lucia.deluca@cnr.it}

\author[A. Diana]
{Antonia Diana}
\address[Antonia Diana]{Dipartimento di Matematica ``Guido Castelnuovo'', Sapienza Universit\`a di Roma, Piazzale Aldo Moro 2, I-00185 Roma, Italy
}
\email[A. Diana]{antonia.diana@uniroma1.it}

\author[M. Ponsiglione]
{Marcello Ponsiglione}
\address[Marcello Ponsiglione]{Dipartimento di Matematica ``Guido Castelnuovo'', Sapienza Universit\`a di Roma, Piazzale Aldo Moro 2, I-00185 Roma, Italy
}
\email[M. Ponsiglione]{ponsigli@mat.uniroma1.it}

\usepackage{hyperref}
\begin{document}
 \maketitle

\begin{abstract}
We prove the existence and the $\tfrac12$-H\"older continuity in time of flat flows for periodic Lipschitz subgraphs, whose evolution is governed by the gradient flow of generalized nonlocal perimeters. 
Moreover, we show that the flat flow satisfies the semigroup property and, as a consequence, the generalized perimeter decreases along the evolution.

Finally, we prove that halfspaces are global minimizers of the generalized nonlocal perimeters and act as attractors for the dynamics.

Our theory covers several generalized perimeters, including fractional and  Riesz-type perimeters (defined on entire periodic subgraphs through suitable renormalization procedures)  and the Minkowski pre-content.
\end{abstract}
\setcounter{tocdepth}{1} 
\tableofcontents
Geometric flows describe the evolution of sets in which each point on the moving boundary evolves according to a specific law involving local or nonlocal quantities related to the shape of the set. In this work we focus on curvature flows, which have a gradient flow structure. 
A classical example is provided by the mean curvature flow \cite{H}, where
the evolution is formally driven by the gradient flow of the Euclidean perimeter, so that the boundary velocity is proportional to its first variation, that is the standard mean curvature.
 This geometric flow has been generalized in several directions, 
notably to weak formulations involving
anisotropic and crystalline perimeters \cite{ATW, CMNP,GP}.

In \cite{ES} a level set approach to the classical mean curvature flow is analyzed; within this framework, the evolution of a set is represented via the superlevel sets of an auxiliary level set function that solves a degenerate parabolic equation.
A variational approach to the mean curvature flow, as well as to its anisotropic and crystalline variants, was proposed in \cite{ATW, LS}. It is based on a minimizing movements scheme, which can be viewed as an implicit Euler discretization of the gradient flow of the perimeter functional with respect to a suitable $L^{2}$-Riemannian structure. The discrete evolution satisfies a H\"older continuity estimate in time, and by applying Ascoli-Arzel\`a compactness arguments, one can pass to the limit as the time step tends to zero, thereby obtaining a H\"older continuous evolution, known as the flat flow.

In \cite{C}, it was shown that this step-by-step minimization can be used to construct level set solutions, connecting the approaches in \cite{ES} and \cite{ATW, LS}.

 More recently, attention has also turned to nonlocal evolutions \cite{Sl}, where
the local perimeters and curvatures are replaced by
  quantities depending on the whole set. The fractional perimeter is a prominent example, and the corresponding evolution is referred to as the fractional mean curvature flow \cite{I,CaSo} (see also \cite{DaLFoMo,DLKP}).
Alongside the fractional perimeter, a class of generalized nonlocal perimeters is analyzed in \cite{CMP} (see also \cite{Vi, ChamGiacoLuss,  CDLNP}). This family consists of functionals satisfying certain structural assumptions; the most relevant is a convexity property known as submodularity. Such  a property ensures that the first variation of a (sufficiently smooth) generalized perimeter, referred to as generalized nonlocal curvature, is monotone with respect to set inclusion. Exploiting this monotonicity, 
\cite{CMP}  establishes that geometric evolutions driven by generalized nonlocal curvatures satisfy the comparison principle; 
in particular, existence and, under suitable regularity assumptions, uniqueness of a viscosity level set solution are provided, extending the results of \cite{C} to this generalized nonlocal framework. 

We stress that, within the level set formulation, existence, uniqueness and regularity properties of the level set function typically translate into properties holding only for almost every level.

In this paper, we revisit the results in \cite{CMP}
to recover qualitative properties of the solutions that are not merely generic with respect to 
the initial set (namely, with respect to the level of the level set function), but are instead satisfied by any flat flow starting from a specific class of initial sets, that is  entire Lipschitz subgraphs (a class that is preserved along the flow). Since we work with entire (non-compact) solutions, this setting does not fall under the assumptions of \cite{CMP}, which focuses on compact sets. Therefore, we first adapt the approach of \cite{CMP} to the non-compact periodic setting.
Notice that the perimeter of unbounded subgraphs is generally infinite, motivating our restriction to a periodic framework, where a notion of surface area per unit cell can be introduced, at least for local perimeters. For nonlocal perimeters, this notion is more subtle. Indeed, scaling arguments suggest that the fractional surface per unit volume of an entire periodic hypersurface is infinite. Our examples of nonlocal perimeters in this framework employ specific renormalization procedures, allowing for a meaningful definition of the fractional perimeter for periodic sets.
On the one hand, our results apply to a broad class of generalized perimeters; on the other hand, they rely on techniques and tools that are available only within our framework of Lipschitz subgraphs. Consequently, most of the results in this paper cannot have a clear extension to evolutions starting from arbitrary initial sets.
For instance, to the best of our knowledge, it is currently unknown whether the fractional mean curvature flow, starting from an initial set of finite fractional perimeter, enjoys H\"older continuity in time.

We now describe our results in more detail. We introduce and analyze flat flows of periodic Lipschitz subgraphs for general nonlocal perimeters. Following \cite{CMP}, a generalized perimeter is a non-negative map $\per$ defined on measurable sets, finite on sufficiently smooth sets (here, sets with Lipschitz boundary), invariant under translations and modifications on negligible sets, lower semicontinuous with respect to $L^1_{\mathrm{loc}}$ convergence, and satisfying the submodularity condition. 

Given a Lipschitz subgraph and a positive time step parameter, the minimizing movements scheme produces a discrete-in-time evolution by iteratively solving incremental minimization problems. 
Using geometric and variational arguments, we establish discrete H\"older-in-time estimates for these solutions. Specifically, we follow \cite{LS} which relies on the fact that the Euclidean perimeter of the evolving set is uniformly bounded and the sets satisfy uniform density estimates. These properties are not known for general nonlocal curvature flows, not even for the fractional mean curvature flow starting from a bounded regular set. In contrast, periodic Lipschitz graphs trivially satisfy density estimates and have uniformly bounded surface per unit cell. This regularity-in-time
allows to use Ascoli-Arzel\`a Theorem in order to pass to the limit as the time step vanishes, obtaining a H\"older continuous in time evolution, that is the flat flow. Remarkably, this flat flow  satisfies the semigroup property so that the generalized nonlocal perimeter decreases in time.

A natural question concerns the asymptotic behavior of the flat flow as $t \to +\infty$. 
While for volume preserving mean curvature flows as well as for surface diffusion dynamics the attractors are compact sets \cite{H2, EsSi, DGKK, DGDKK, JMPS, JMOS}, in our periodic subgraph setting the natural candidates are periodic halfspaces.
In Proposition \ref{hypermin}, we show that periodic halfspaces minimize any generalized perimeter $\per$ within the class of periodic Lipschitz subgraphs. This result is new even for fractional perimeters: while the local minimality of halfspaces was established in \cite{P} (see also \cite{Cab, CiSeVa}), a global minimality based on a natural notion of fractional perimeter for periodic subgraphs was still missing.

Given their minimality, halfspaces are natural candidate attractors for the dynamics. If halfspaces are the only critical points of $\per$ within Lipschitz subgraphs, compactness arguments show that the discrete evolution converges to a halfspace as $t \to +\infty$. In Proposition \ref{quantiver}, we prove that a quantitative version of this assumption guarantees convergence of the flat flow to a halfspace. Such a result is classical for the Euclidean perimeter \cite{EH, W} and has been recently extended to fractional perimeters in \cite{CN}, where the asymptotic behavior of sets near entire cones, along with H\"older continuity properties of the solution, was established in the level set framework. On the one hand, our results extend \cite{CN} to generalized nonlocal perimeters, providing H\"older continuity of the flat flow (with exponent $1/2$) rather than of the level set function; on the other hand, our framework is restricted to the periodic setting.

Finally, we present several classes of generalized perimeters to which our abstract results apply. First, fractional perimeters \cite{Vi,CRS2010} arising from nonlocal interactions, defined via double integrals with power-law kernels. We present two natural extensions to periodic entire sets: the first one is defined
 on subsets of the $(d-1)$-dimensional torus times $\R$ and the interactions are modulated by the classical fractional kernel involving the Riemannian distance on the torus;
the second one is defined on periodic subsets of $\R^d$ endowed with the Euclidean distance, using a renormalization procedure that subtracts the infinite tail of a reference halfspace. We then extend the fractional framework to Riesz-type perimeters and to the $0$-fractional perimeter \cite{DNP}, as well as to the Minkowski pre-content \cite{CMP}. These examples demonstrate the flexibility of our scheme and the broad applicability of our results.

The paper is organized as follows. In Section \ref{preliminari} we introduce notation, the formal definition of generalized nonlocal perimeter, the discrete scheme and continuity properties for discrete solutions. Section \ref{existence} presents the main existence theorem for the flat flow. 
Section \ref{examples} illustrates examples of generalized perimeters covered by our theory and
Section \ref{convergencesec} addresses the asymptotic convergence to periodic halfspaces.


\vskip10pt
\textsc{Acknowledgements:} LDL and AD are members of the Gruppo Nazionale per l'Analisi Matematica, la Probabilit\`a e le loro Applicazioni (GNAMPA) of the Istituto Nazionale di Alta Matematica (INdAM).

LDL acknowledges the financial support of PRIN 2022HKBF5C ``Variational Analysis of complex systems in Materials Science, Physics and Biology'', PNRR Italia Domani, funded by the European Union via the program NextGenerationEU, CUP B53D23009290006.

MP acknowledges the financial support of PRIN 2022J4FYNJ ``Variational methods for stationary and evolution problems with singularities and interfaces'', PNRR Italia Domani, funded by the European Union via the program NextGenerationEU, CUP B53D23009320006.

Views and opinions expressed are however those of the authors only and do not necessarily reflect those of the European Union or the European Research Executive Agency. Neither the European Union nor the granting authority can be held responsible for them.

\section{The minimizing movements scheme}\label{preliminari}

We start by introducing the notion of  periodic Lipschitz subgraphs on the torus.

Let $d\in\N$ with $d\ge 2$.
We denote by $\T^{d-1}:=\R^{d-1}/\Z^{d-1}$ the $(d-1)$-dimensional {flat} torus.
With a little abuse of notations, translations in $\T^{d-1}\times \R$ are identified with vectors in $\R^d=\R^{d-1}\times\R$.
Moreover, for every $R>0$, we set $\S^d_R:= \T^{d-1}\times (-R,R)$.
\begin{definition}\label{lipgraphdef}
Let $E\subset \T^{d-1} \times \R$. We say that $E$ is a periodic $L$-Lipschitz subgraph (with $L>0$) if there exists a $L$-Lipschitz continuous function $g: \T^{d-1} \to \R$  such that
$$
E=\{ (x', x_d) \in \T^{d-1} \times \R \, : \, x_d \leq g(x')\}  .
$$
In the following, given $L>0$,  we denote by $\Reg_L$ the class of periodic $L$-Lipschitz subgraphs. 
\end{definition}
Now, following the approach in \cite{CMP}, we introduce the notion of generalized perimeters for measurable subsets of $\T^{d-1} \times \R$. 
\begin{definition}\label{per}
Let $\Ins$ be the class of measurable subsets of $\T^{d-1} \times \R$. We will say that a functional  $\per:\Ins\to [0,+\infty]$ is a generalized perimeter if it satisfies the following properties:

\begin{align}\label{iprop}\tag{{i}}
&\textrm{$\per(E)<+\infty$  for every $E\in\Reg_L$;}
\\ \label{iiprop}\tag{{ii}}
&\textrm{$\per(\emptyset)=\per(\T^{d-1} \times \R)=0$; }
\\ \label{iiiprop}\tag{{iii}}
&\textrm{$\per(E)=\per(E')$ if $|E\triangle E'|=0$; }
\\ \label{ivprop}\tag{{iv}}
&\textrm{$\per$ is $L^1_{\loc}$-l.s.c.: if $|(E_n\triangle E)\cap \S^d_R \big)|\to 0$ for every $R>0$, then
$\per(E)\le \liminf_{n\to +\infty} \per(E_n)$;}
\\ \label{vprop}\tag{{v}}
&\textrm{$\per$ is { submodular}: For any $E, \, F\in\Ins$: }
\per(E\cup F)\,+\, \per(E\cap F)\ \le\ \per(E)\,+\,\per(F)\ ;
\\ \label{viprop}\tag{{vi}}
&\textrm{$\per$ is translational invariant: }
\per( E+ \tau)\ =\ \per(E)\quad \text{ for all } E \in \Ins ,\,  \tau \in  \R^d .
\end{align}
\end{definition}

\begin{remark}\label{chgilu}
Given a generalized perimeter $\per$ in the sense of Definition \ref{per}, we can extend it to $L_{\loc}^1(\T^{d-1} \times \R)$ 
enforcing the following {\it generalized co-area formula}:
\begin{equation}\label{visintin}
{\per}(u) :=\ \int_{-\infty}^{+\infty}\per(\{u>s\})\ud s \qquad \text{ for every } u\in L_{\loc}^1(\T^{d-1} \times \R).
\end{equation}
It can be
shown that, under the assumptions above,  $\per$ is  a convex and lower semicontinuous (with respect to vanishing $L^1$ perturbations having uniformly compact support) functional in $L_{\loc}^1(\T^{d-1} \times \R)$ (see~{\cite[Proposition~3.4]{ChamGiacoLuss}} for a detailed proof in $\R^d$).  
\end{remark}

Now we define the minimizing movement scheme for the generalized perimeter $\per$, following the approach in  \cite{ATW} and \cite{LS} (see also \cite{CMP}).

Let $R>0$ be fixed; we set
$$
\Ins_R:= \{E\in \Ins : \T^{d-1} \times (-\infty,-R) \subseteq E \subseteq   \T^{d-1}  \times (-\infty,R) \}.
$$  
From now on, we will always assume that any set  $E$ in $\Ins_R$ coincides with its Lebesgue representative $\tilde E:=\{x\in E: x \text{ is a point of density one for } E\}$.   

Let $h>0$ be the time step of the approximation scheme. For any $E,\, F\in \Ins_R$, we define
\begin{equation}\label{F}
\F_h(E,F):=\per(F)+ \frac{1}{h}  \int_{E \Delta F} \dist(x,\partial E) \ud x,    
\end{equation}
where $\dist( \cdot, \partial E)$ denotes the standard distance function from the boundary of $E$. 
\\Moreover, for any $E \in \Ins_R$, we define the signed distance from $E$ as
$$
\sdist_E(x):=   \dist(x,E) - \dist(x,E^c)  .
$$
Notice that 
\begin{equation*}
\int_{F\cap \S^d_R} \sdist_E(x) \ud x =  \int_{F \Delta E} \dist(x,\partial E) \ud x - \int_{E\cap \S^d_R} \dist (x,\partial E) \ud x\qquad\textrm{for any }E, F\in \Ins_R;
\end{equation*}
as a consequence, the class of minimizers of $\F_h(E,\cdot)$ in $\Ins_R$ 
coincides with that of the functional $\widetilde{\F}_h(E,\cdot)$ (in $\Ins_R$) defined, for every $F\in\Ins_R$, as
\begin{equation}\label{Ftilde}
\widetilde{\F}_h(E,F):=\per(F)+ \frac{1}{h}  \int_{F\cap \S^d_R} \sdist_E(x) \ud x.
\end{equation}
Notice that the volume term in \eqref{Ftilde} satisfies the submodularity inequality (with respect to $F$) with an equality; as a consequence, using the submodularity  \eqref{vprop} of $\per$, for every $F,F'\in\Ins_R$ we have
\begin{equation}\label{Ftildesub}
\widetilde{\F}_h(E,F\cup F')+\widetilde{\F}_h(E,F\cap F')\le\widetilde{\F}_h(E,F)+\widetilde{\F}_h(E,F').
\end{equation}
In the remainder of this section, we prove that, given $h>0$, for any given  set $E\in\Ins_R$ the problem
\begin{equation}\label{probmin0}
\min_{F\in\Ins_R} \widetilde{\F}_h(E,F)
\end{equation}
admits at least a solution in $\Ins_R$.
To this end we preliminarily show that ``periodic hyperspaces'' minimize $\per$ in $\Ins_R$.
\begin{proposition}\label{hypermin}
For every $\lambda\in (-R,R)$, the sets $H_\lambda:= \T^{d-1}\times (-\infty,\lambda)$ minimize $\per$ in $\Ins_R$.
\end{proposition}
\begin{proof}
First, we show that there exists a minimizer of $\per$ in $\Ins_R$.
Since, by Remark \ref{chgilu},  the functional $\per$ extended to $L^1_{\loc} (\T^{d-1} \times \R)$ is convex  and lower semicontinuous, 
there exists a solution $\bar u$ to the minimization problem 
\begin{equation}\label{minfunzioni}
\min_{\newatop{u \in L^\infty (\T^{d-1} \times \R; [0,1])}{\newatop{u=1\textrm{ in }\T^{d-1}\times (-\infty,-R)}{u=0\textrm{ in }\T^{d-1}\times (R,+\infty)} }} \per(u).
\end{equation}
Since
\begin{equation*}
\per(\bar u) = \int_0^1 \per(\{\bar u > s\})\ud s,
\end{equation*}
one can easily deduce that, for almost every $s \in (0,1)$, the set $\{\bar u>s\}$ is a minimizer of $\per$ in $\Ins_R$. 


Now we notice that, given $E\in\Ins_R$, for every translation $\tau\in\R^d$, 
the sets $E\cup (E+\tau)$ and $E\cap (E+\tau)$ belong, up to a translation, to $\Ins_R$;
moreover, by the submodularity \eqref{vprop} and by the translational invariance \eqref{viprop}  of $\per$, we have
$$
\per(E\cup (E+\tau))+\per(E\cap (E+\tau))\le 2\per (E).
$$
Hence, if $E$ is a minimizer of $\per$ in $\Ins_R$, it holds
\begin{equation}\label{alleq}
\per(E\cap (E+\tau))=\per(E\cup(E+\tau))=\per(E).
\end{equation}
Let $E$ be a minimizer of $\per$ in $\Ins_R$ and let $\{\tau_n\}_{n\in\N}$ be a dense sequence in $\T^{d-1}\times\{0\}$. 
We set $E_0:=E$ and $E_n:=E_{n-1}\cup (E_{n-1}+\tau_n)$, for every $n\in\N$. Then, by an induction argument using \eqref{alleq}, we have that $\per(E_n)=\per(E)$.  By construction, the sequence $\{E_n\}_{n\in\N}$ converges (in $L^1_\loc$) to the set $E_\infty:=\bigcup_{n\in\N}E_n$. 
\\We conclude by noticing that on the one hand, by the lower semicontinuity \eqref{ivprop} of $\per$, 
$$\per(E_\infty)\le\per(E) ,$$
on the other hand, by construction, $E_\infty=H_{\lambda_{\min}}$ where $\lambda_{\min}:=\inf\{\lambda: H_\lambda\supset E\}$.
By the translational invariance \eqref{viprop} of $\per$, we deduce that all the sets $H_\lambda$ (for $\lambda\in (-R,R)$) minimize $\per$ in $\Ins_R$.
\end{proof}
In what follows, we will adapt what was done in \cite{CMP} for bounded sets.

%
\begin{lemma}\label{lemmamin}
Let $E \in \Ins_R$ be a set with $\per(E)<+\infty$.
For every $h>0$, there exist a minimal $T_h^-[E]$ and a maximal $T_h^+ [E]$ (with respect to inclusion) solution to \eqref{probmin0}. 
Moreover, $T_h^\pm [E] \in \Ins_{R_{\min}}$ where $R_{\min}:= \inf \{r>0: E\in \Ins_r\}$.  
\end{lemma}
\begin{proof}
Since, by Remark \ref{chgilu}, the functional $\per$ extended to $L^1_{\loc} (\T^{d-1} \times \R)$ is convex and lower semicontinuous, 
there exists a solution $\bar u$ to the minimization problem 
\begin{equation}\label{minfunzioni}
\min_{\newatop{u \in L^\infty (\T^{d-1} \times \R; [0,1])}{\newatop{u=1\textrm{ in }\T^{d-1}\times (-\infty,-R)}{u=0\textrm{ in }\T^{d-1}\times (R,+\infty)} }} \per(u) + \frac 1h \int_{\S^{d}_R} u(x) \,\sdist_E(x) \ud x.
\end{equation}
By the layer cake formula,
\begin{equation*}
\per(\bar u) + \frac 1h \int_{\S^{d}_R} \bar u(x) \,\sdist_E(x) \ud x = \int_0^1 \left (\per(\{\bar u > s\}) + \frac 1h \int_{\S_R^{d}\cap\{\bar u>s\}} \sdist_E(x) \ud x \right)\ud s,
\end{equation*}
whence, one can easily deduce that, for almost every $s \in (0,1)$, the set $F^s:=\{\bar u>s\}$ is a minimizer for \eqref{Ftilde}. 
\vskip2pt
We only show the existence of a minimal solution, since the existence of a maximal solution can be proven analogously. 

Let us define 
$$m:= \inf \{|F\cap \S^d_R| \, : F \in \Ins_R\text{ and $F$ minimizes $\widetilde \F_h (E, \cdot)$}\}$$
 and consider a minimizing sequence $\{F_n\}_{n\in\N}$, in the sense that $F_n$ is a solution to \eqref{Ftilde} (for every $n\in\N$) and
$$|F_n\cap\S^d_R| \to m.$$
For every $N\in\N$ we set $F^{(N)}:=\bigcap_{n=1}^{N}F_n$. Then, by \eqref{Ftildesub}, $\{F^{(N)}\}_{N\in\N}$ is  a decreasing sequence (with respect to the inclusion) of solutions to \eqref{probmin0}, satisfying  $|F^{(N)}\cap\S^d_R | \to m$, as $N \to \infty$.
Then, setting $F^{(\infty)}:=\bigcap_{n=1}^{+\infty}F_n$, we have that $F^{(N)}$ converges (in $L^1_\loc$) to $F^{(\infty)}$; by lower semicontinuity \eqref{ivprop} of $\per$ and by the continuity of the dissipation with respect to the strong $L^1$ convergence in $\S^d_R$, we obtain that $F^{(\infty)}$ is a minimal solution.
\vskip2pt
Finally, we prove the last sentence of the statement. 

To this end, we first claim that, if $H_\lambda\subset E$ for some $\lambda\in(-R,R)$ and 
\begin{equation}\label{distpos}
\dist(\partial E,\partial H_\lambda)=:\delta_0>0,
\end{equation}
then, any minimizer $F$ to \eqref{probmin0} satisfies
\begin{equation}\label{claimH}
|H_\lambda\setminus F|=0  ,
\end{equation}
that is, $H_\lambda\subset F$. 

In order to prove \eqref{claimH}, we first observe that, by \eqref{distpos},
\begin{equation}\label{distord}
\sdist_E\le \sdist_{H_\lambda}-\delta_0<\sdist_{H_\lambda}.
\end{equation}
By definition of $F$, we have
\begin{equation}\label{minperE}
\per(F)+\frac{1}{h}\int_{F\cap\S^d_R}\sdist_E\ud x\le\per(F\cup H_\lambda)+\frac{1}{h}\int_{(F\cup H_\lambda)\cap\S^d_R}\sdist_E\ud x
\end{equation}
and, by Proposition \ref{hypermin}, we get
\begin{equation}\label{minperH}
\per(H_\lambda)+\frac{1}{h}\int_{H_\lambda\cap\S^d_R}\sdist_{H_\lambda}\ud x\le\per(F\cap H_\lambda)+\frac{1}{h}\int_{(F\cap H_\lambda)\cap\S^d_R}\sdist_{H_\lambda}\ud x.
\end{equation}
By summing \eqref{minperE} and \eqref{minperH}, and using the submodularity \eqref{vprop} of $\per$, we obtain
\begin{equation*}
\int_{H_\lambda\setminus F}\sdist_{H_\lambda}\ud x\le \int_{H_\lambda\setminus F}\sdist_{E}\ud x,
\end{equation*}
which, in view of \eqref{distord}, is satisfied if and only if \eqref{claimH} holds true.
Using an approximation argument one can prove that the claim holds true also without assuming \eqref{distpos}.
Analogously one can show that $H_\lambda\supset F$ whenever $H_\lambda\supset E$.
Therefore, if $E\in\Ins_r$ for some $r>0$, then every minimizer $F$ for \eqref{probmin0} satisfies $F\in\Ins_r$. 
Passing to the infimum with respect to $r$, we get that $T_h^\pm E\in \Ins_{R_{\min}}$.
\end{proof}
Given two sets $E,E' \in \Ins_R$, we write that $E \subset\subset E'$ if
$\sdist_E>\sdist_{{E'}}$.
\begin{lemma}\label{lemmamonot}
Let $E,E' \in \Ins_R$. If  $E \subset\subset E'$, then  $T_h^+[E] \subset\subset T_h^- [E'] $; if   $E \subseteq E'$,  then, $T_h^\pm [E ]\subseteq T_h^\pm [E'] $.
\end{lemma}
\begin{proof}
Assume first that $E \subset \subset E'$ and let $\widehat E\in\Ins_R$ be such that $E \subset \subset \widehat{E} \subset \subset E'$; then 
\begin{equation}\label{wellcont}
\sdist_E(x)>\sdist_{\widehat{E}} (x) \qquad\textrm{for every }x \in \S^d_R.
\end{equation}
Comparing $\widetilde{\F}_h(E,T^+ _h [E])$ with $\widetilde{\F}_h(E,T^+ _h [E]\cap T^-_h [\widehat{E}])$ and $\widetilde{\F}_h(\widehat{E},T^- _h[\widehat{E}])$ with $\widetilde{\F}_h(E,T^+ _h [E]\cup T^-_h[ \widehat{E}])$, and  using the submodularity \eqref{vprop} of $\per$, we get
\begin{equation*}
\int_{(T^+_h [E] \setminus T^-_h [\widehat{E}])  \cap \S^d_R}  \sdist_E(x) \ud x \leq \int_{(T^+ _h [E] \setminus T^- _h [\widehat{E}] )\cap \S^d_R}  \sdist_{\widehat{E}}(x) \ud x.
\end{equation*}
Hence, from \eqref{wellcont}, it follows that $T^+ _h [E]\subset T^-_h [\widehat E]$.
\\To conclude the proof we use a perturbation argument. For every $\ep>0$, we denote by $F_\ep$ the minimal solution of \eqref{probmin0} with $\sdist_E$ replaced by $\sdist_E+\ep$. This perturbation is equivalent to replacing $E$ with its $\varepsilon$-shrunken version; hence, we have $F_\ep$ are decreasing in $\ep$, $F_\ep \subseteq T^-_h E'$ and $F_\ep \to F_0= \bigcup_\ep F_\ep$ in $L^1_{\loc}(\R^d)$. By
lower semicontinuity \eqref{ivprop} of $\per$,  it follows that $F_0$ is a solution, and thus $T^-_h E \subseteq F_0 \subseteq T^-_h E'$. One can prove the same inclusion for $T^+_h E \subseteq T^+_h E'$.
\end{proof}
\begin{remark}
As an immediate consequence of minimality, we get the following result.
Let $E \in \Ins_R$ be a set such that $\per(E_0)<+\infty$ and let $T_h^\pm [E]$ be the minimizers given in Lemma \ref{lemmamin}. Then, $T_h^\pm [E]$ satisfy the discrete dissipation inequality
\begin{equation}\label{dissineq}
\per(T_h^\pm [E]) + \frac1h \int_{E\Delta T_h^\pm [E]}\dist(x, \partial E) \ud x \leq \per(E)  .
\end{equation}
\end{remark}

\begin{proposition}\label{preserLip}
Let $E \in \Reg_L\cap \Ins_R$. Then, $T_h^- [E] = T_h^+ [E]=: T_h[E]$ and $T_h[E]\in \Reg_L$.
\end{proposition}

\begin{proof}
We notice that a set $F$ is in $\Reg_L$ if and only if
\begin{equation}\label{charac}
F+ \tau  \subseteq F\qquad\textrm{for every }\tau=(\tau', -L|\tau'|)\in\R^d.
\end{equation}

Let $T_h^-[E]$ be the minimal solution to \eqref{probmin0}.
Given $\tau\in\R^d$ as in \eqref{charac},  by translational invariance \eqref{viprop},
by (the second part of) Lemma \ref{lemmamonot} and, using that $E \in \Reg_L$  (so that $E+\tau\subseteq E$), 
we
conclude  that
$$T_h^-[E] + \tau= T^-_h[E+ \tau  ]\subseteq T^-_h[E].
$$
By the arbitrariness of $\tau$ as in \eqref{charac}, we deduce that $T^-_h[E]\in \Reg_L$.
Analogously,  one can prove that $T^+_h[E]\in \Reg_L$.

Since $E\in\Reg_L$, we have that  $E\subset \subset E+\ep e_d$ for every $\ep>0$, so that by (the first part of) Lemma \ref{lemmamonot} and by translational invariance \eqref{viprop}, we obtain
$$
T_h^-[E]\subseteq T_h^+[E]\subseteq T_h^-[E+\ep e_d]=T_h^-[E]+\ep e_d,
$$
whence, sending $\ep\to 0$, we deduce $T_h^+[E]=T_h^-[E]$.
\end{proof}
\begin{remark}
We notice that for every $E \in \Reg_L$, there exists $R=R(L)>0$ such that $E \in \Ins_R$. In this respect, the assumption $E\in\Reg_L\cap\Ins_R$ is a bit redundant.
\end{remark}

\begin{definition}
Let $E_0 \in \Reg_L$ be a set such that $\per(E_0)<+\infty$ and $h>0$.
\\We iteratively define the sequence of sets $\{\Eh_{kh}\}_{k\in\N}$ by setting
\begin{equation}\label{approxflatflow}
\Eh_0:=E_0\quad\text{and}\quad \Eh_{kh}:= T_h \Eh_{(k-1)h} \quad\text{for }k\ge1.
\end{equation}
We furthermore define
$$
\Eh_t:=\Eh_{kh}\quad\text{for any }t\in[kh,(k+1)h),
$$
and call $\{\Eh_t\}_{t\geq 0}$ an {\em approximate flat solution}
to the generalized curvature flow, associated to the perimeter $\per$, with initial datum $E_0$ and time step equal to $h$.
\end{definition}

\begin{remark}\label{regolarita}
If $E_0\in \Ins_R\cap\Reg_L$, by Lemma \ref{lemmamin} and by Proposition \ref{preserLip}, we get $\Eh_t \in \Ins_R\cap\Reg_L$ for every $t\geq 0$.
\end{remark}
As an immediate consequence of the dissipation inequality \eqref{dissineq}, we have for every $t \geq h$
\begin{equation}\label{monot}
\per(\Eh_t) + \frac 1 h  \int_{\Eh_t \Delta \Eh_{t-h}} \dist (x,\partial {\Eh_{t-h}}) \ud x \leq \per(\Eh_{t-h}) 
\end{equation}
and, by induction, 
\begin{gather}
\per(\Eh_t ) \leq \per(E_0) \qquad \forall\,  t>0 \label{stimaper}
\\ \frac 1 h \int_h^T \int_{\Eh_t \Delta \Eh_{t-h}} \dist_{\Eh_{t-h}} (x) \ud x \ud t \leq \per(E_0)\qquad \forall  \, T >h . 
\end{gather}

\begin{definition}\label{rhoint}
Let $E\in \Reg_L$,  for every $\rho>0$ we define the $\rho$--neighborhood of $\partial E$ as
\begin{equation}\label{ingr}
\fat_\rho(\partial E):= \{ x \in \T^{d-1} \times \R \, : \, \dist(x, \partial E) \leq \rho \}  .
\end{equation}
\end{definition}
\begin{proposition}[H\"older continuity in time]\label{holdercont}
Let $E_0 \in \Reg_L$ with $\per(E_0)<+\infty$, $0<h< 1$ and $\{E_t\}_{t \geq 0}$ be an approximate flat flow with initial datum $E_0$. 
Then it holds
\begin{equation*}
\vert\Eh_t\Delta\Eh_{s}\vert \leq C \max\{h^{1/2},\vert t-s\vert^{1/2}\}
\qquad\textrm{for every }s,t\ge 0,
\end{equation*}
for some constant $C$ depending on $L$ and $E_0$.
\end{proposition}
\begin{proof}
For every $i \in \N$ and $\rho>0$, we define
\begin{equation}
G^h_i:=\big(\Eh_{(i+1)h} \setminus \Eh_{ih}\big) \cap \fat_\rho (\partial \Eh_{ih}) \quad \text{and} \quad F^h_i := \big(\Eh_{(i+1)h} \setminus \Eh_{ih}\big)\setminus \fat_\rho (\partial \Eh_{ih})
\end{equation}
and we notice that 
\begin{equation}\label{insiemi}
\Eh_{(i+1)h}\setminus \Eh_{ih}=G^h_i \cup F^h_i.
\end{equation}
Since $G^h_i \subseteq  \fat_\rho (\partial \Eh_{ih})$, from Proposition \ref{preserLip} we get
\begin{equation}\label{stimaG}
| G^h_i | \leq C_L \rho. 
\end{equation}
Moreover, from the dissipation inequality in \eqref{dissineq} we have
\begin{equation*}
\per(\Eh_{ih})-\per(\Eh_{(i+1)h})  \geq \frac 1h \int_{F^h_i} \dist (x,\partial \Eh_{ih}) \ud x \geq \frac \rho h |F^h_i|  ,
\end{equation*}
that is
\begin{equation}\label{stimaF}
|F^h_i| \leq \frac h \rho  \left[ \per(\Eh_{ih})-\per(\Eh_{(i+1)h}) \right]  .
\end{equation}
By \eqref{insiemi}, \eqref{stimaF} and \eqref{stimaG},
\begin{equation}\label{primadiff}
|\Eh_{(i+1)h}\setminus \Eh_{ih}|\le \frac h \rho  \left[ \per(\Eh_{ih})-\per(\Eh_{(i+1)h}) \right] +C_L \rho,
\end{equation}
and analogously one can show
\begin{equation}\label{secodiff}
|\Eh_{ih}\setminus (\Eh_{(i+1)h}|\le \frac h \rho  \left[ \per(\Eh_{ih})-\per(\Eh_{(i+1)h}) \right] +C_L \rho.
\end{equation}
Without loss of generality, we take $s= ih$ with $i \in \N\cup\{0\}$  and $t= (i+k) h$, with $k \in \N $, that is $\Eh_t=\Eh_{ih}$ and $\Eh_s= \Eh_{(i+k)h}$, respectively. Hence, by \eqref{primadiff} and \eqref{secodiff}, we get
\begin{equation}
\begin{aligned}
| \Eh_t \Delta \Eh_s | \leq & \sum_{l=0}^{k-1} | \Eh_{(i+l)h} \Delta \Eh_{(i+l+1)h}|   
\\ \leq &2  \frac h \rho  \left[ \per(\Eh_{ih})-\per(\Eh_{(i+k)h}) \right] +2 C_L k \rho  
\\ \leq & 2\frac h\rho  \per(E_0)+2C_L k \rho  .  
\end{aligned}
\end{equation}
By choosing $\rho= \sqrt \frac hk$, we conclude
\begin{equation*}
| \Eh_t \Delta \Eh_s | \leq  C_{L,E_0}\max\{h^{\frac 1 2},|t-s| ^{1/2}\}.
\end{equation*}
\end{proof}
\section{Existence of flat solutions}\label{existence}
We are in a position to establish the existence of a H\"oder continuous flat flow, obtained as a limit of the discrete minimizing movements as the time step $h$ vanishes.

  \begin{theorem}[Flat flows]\label{mainteo1}
Let $E_0 \in\Ins_R\cap \Reg_L$ with $\per(E_0)<+\infty$ and let $\{h_n\}_{n\in\N} \subset (0,+\infty)$ be a monotonically vanishing sequence. Moreover, for any $h_n$, let $\{E^{(h_n)}_t[E_0]\}_{t \geq 0}=\{E^{(h_n)}_t\}_{t \geq 0}$ be the approximate flat solution
to the generalized curvature flow, associated to the perimeter $\per$, with initial datum $E_0$. Then, there exist 
 a 
 subsequence $\{h_{n_k}\}_{k\in\N}=:\{h_k\}_{k\in\N}$
 and a family
$\{E_t^0[E_0,\{h_k\}_k]\}_{t\geq 0} = \{E^0_t\}_{t\geq 0} \subset \Reg_L$ with $E_0^0=E_0$
such that, for all $t\ge 0$
\begin{equation}\label{limite0}
|E_t^{(h_k)} \Delta E^0_t| \to 0, 
\end{equation}
and, for every $ t,s\ge 0$, 
\begin{equation}\label{holderflat}
\vert E_t^0\Delta E_s^0\vert\leq C \vert t-s\vert^{1/2}. 
\end{equation}
Finally, the {\it flat flow} $\{E^0_t\}_{t\ge 0}$ satisfies the so called semi-group property: for every $\bar t,\Delta t\ge 0$ it holds
\begin{equation}\label{semig}
E^0_{\Delta t}[E^0_{\bar t}[E_0,\{h_k\}_k], \{h_k\}_k]=E^0_{\bar t+\Delta t}[E_0,\{h_k\}_k].
\end{equation}
As a consequence, $\per (E^0_t)$ is monotonically non-increasing (with respect to $t$), i.e.,
\begin{equation}\label{monotone}
\per (E^0_{t'})\leq \per(E^0_{t})\qquad\textrm{for every }0\le t\le t'. 
\end{equation}
\end{theorem}
\begin{proof}
Let $t\in\mathbb Q^+$ be fixed. 
By Remark \ref{regolarita} we can apply Ascoli-Arzel\'a Theorem to deduce that there exist a subsequence $\{h^t_{n_k}\}_k$ of $\{h_n\}_n$ and a set $E^0_t$ such that
\begin{equation}\label{limite}
\lim_{k \to \infty} | E^{(h^t_{n_k})}_t \Delta E^0_t| =0.
\end{equation}
By a standard diagonal argument, \eqref{limite} holds true for any $t\in\mathbb Q^+$ with a subsequence  $\{h_{n_k}\}=\{h_k\}$ 
independent of $t$.
Moreover, let  $s, t \in \mathbb Q^+$;
using the triangular inequality, \eqref{limite}, and the H\"older continuity in Proposition \ref{holdercont}, we get
\begin{equation*}
| E^0_s \Delta E^0_t| \leq \limsup_{k \to \infty} \left( | E^{(h_k)}_t \Delta E^0_t| + |E^{(h_k)}_t \Delta E^{(h_k)}_s | + |E^{(h_k)}_s \Delta E^0_s| \right) \leq C |t-s|^{1/2}.
\end{equation*}
By completeness of $L^1$, we can extend by continuity the maps $t\mapsto E^0_t\cap \S_R^d$ and hence (adding $(\T^{d-1}\times(-\infty,-R))$) the maps $t\mapsto E^0_t$, from $\mathbb Q^+$ to $\R^+$.
 By construction, 
$E^0_t \in \Reg_L$ for every $t\ge 0$ and satisfy the H\"older continuity property \eqref{holderflat}.
\\Moreover, for every $t\ge 0$ and for every $\hat t\in Q^+$, by triangular inequality and Proposition \ref{holdercont}, we have
\begin{align*}
|E^{(h_k)}_t\Delta E^0_{t}|\le& |E^{(h_k)}_t\Delta E^{(h_k)}_{\hat t}|+|E^{(h_k)}_{\hat t}\Delta E^0_{\hat t}|+|E^0_{\hat t}\Delta E^0_{t}|
\\ \le &C|t-\hat t|^{1/2}+|E^{(h_k)}_{\hat t}\Delta E^0_{\hat t}|+|E^0_{\hat t}\Delta E^0_{t}|,
\end{align*}
whence, using \eqref{limite} and sending $\hat t\to t$, we get \eqref{limite0}.

Now we prove that for every $\bar t,\Delta t\ge 0$ property \eqref{semig} holds true.
Let $\ep>0$. By \eqref{limite0} and Proposition \ref{preserLip},  there exists $k_\ep\in\N$ such that
\begin{equation*}
E^{(h_k)}_{\bar t}-\ep e_d\subseteq E^0_{\bar t}[E_0,\{h_k\}_k]=E^0_{\bar t}\subseteq E^{(h_k)}_{\bar t}+\ep e_d\qquad\textrm{for }k\ge k_\ep,
\end{equation*}
whence, using Lemma \ref{lemmamonot} and translational invariance \eqref{viprop}, we deduce
\begin{equation}\label{confrobar}
E^{(h_k)}_{\bar t+\Delta t}-\ep e_d\subseteq E^{(h_k)}_{\Delta t}[E^0_{\bar t}]\subseteq E^{(h_k)}_{\bar t+\Delta t}+\ep e_d \qquad\textrm{for }k\ge k_\ep.
\end{equation}
Furthermore, by the first part of the statement, we have that there exists a subsequence $\{h_{k_j}\}_{j\in\N}=:\{h_j\}_{j\in\N}$, depending on $E^0_{\bar t}=E^0_{\bar t}[E_0,\{h_k\}_{k}]$, such that 
$$
|E_{\Delta t}^{(h_j)}[E^0_{\bar t}]\Delta E_{\Delta t}^{0}[E^0_{\bar t},\{h_j\}_j]|\to 0\qquad\textrm{as }j\to+\infty,
$$
so that, using again Proposition \eqref{preserLip}, for $j$ large enough,
\begin{equation}\label{comelimite0}
E_{\Delta t}^{(h_j)}[E^0_{\bar t}]-\ep e_d\subseteq E_{\Delta t}^{0}[E^0_{\bar t},\{h_j\}_j]\subseteq E_{\Delta t}^{(h_j)}[E^0_{\bar t}]+\ep e_d
\end{equation} 
By combining \eqref{confrobar} with \eqref{comelimite0}, we deduce that for $j$ large enough
\begin{equation*}
E^{(h_j)}_{\bar t+\Delta t}-2\ep e_d\subseteq E_{\Delta t}^{0}[E^0_{\bar t},\{h_j\}_j]\subseteq E^{(h_j)}_{\bar t+\Delta t}+2\ep e_d,
\end{equation*} 
which, sending $j\to +\infty$ and using again \eqref{limite0} with $t=\bar t+\Delta t$, yields that
\begin{equation*}
E^{0}_{\bar t+\Delta t}[E_0,\{h_k\}_k]-2\ep e_d\subseteq E^{0}_{\Delta t}[E^0_{\bar t}[E_0,\{h_k\}_k],\{h_j\}_{j}]\subseteq E^{0}_{\bar t+\Delta t}[E_0,\{h_k\}_k]+2\ep e_d,
\end{equation*}
so that, by the arbitrariness of $\ep$
\begin{equation}\label{nosub}
E^{0}_{\Delta t}[E^0_{\bar t}[E_0,\{h_k\}_k],\{h_j\}_{j}]=E^{0}_{\bar t+\Delta t}[E_0,\{h_k\}_k].
\end{equation}
Now, since the right-hand side of the identity in \eqref{nosub} is independent of the subsequence $\{h_j\}_j$, we have that also the left-hand side should be, whence we deduce that \eqref{semig} holds true.

Finally, let $0\le t\le t'$. By \eqref{semig} (applied with $\bar t=t$ and $\Delta t=t'-t$), using the lower-semicontinuity property \eqref{ivprop} of $\per$ and inequality
 \eqref{stimaper} applied with $E_0=E^0_t[E_0,\{h_k\}_k]$, we have
\begin{equation*}
\begin{aligned}
\per(E^0_{t'})=&\,\per(E^0_{t'}[E_0,\{h_k\}_k])=\per(E^{0}_{t'-t}[E^0_t[E_0,\{h_k\}_k],\{h_k\}_k]) )\\
 \leq&\, \liminf_{k \to \infty} \per (E^{(h_k)}_{t'-t}[E^0_t[E_0,\{h_k\}_k],\{h_k\}_k])\\
  \leq&\, \per(E^0_t[E_0,\{h_k\}_{k}])=\per(E^0_t),
\end{aligned}
\end{equation*}
which proves \eqref{monotone}.
\end{proof}
\section{Examples of generalized perimeters}\label{examples}

In this section we list some examples of {\em generalized perimeters} fitting the assumptions \eqref{iprop} - \eqref{viprop}.  The first example is given by the standard De Giorgi perimeter, together with its anisotropic, possibly crystalline, variants. Next, we focus on non-local perimeters.
\vskip5pt
{\bf Fractional perimeters.}
Another relevant class is given by  fractional type perimeters \cite{CRS2010}. For the sake of completeness, we introduce a quite general class of nonlocal perimeters governed by convolution kernels, and check that, under suitable assumptions, they also fit in our theory.   

Let $\|\cdot\|_{\sharp}:\T^{d-1}\times\R\to[0,+\infty)$ be the ``periodic norm'' defined by
\begin{equation}\label{metric}
\|\xi\|_{\sharp}:=\big(|\xi_d|^2+\min_{z'\in\Z^{d-1}}|\xi'+z'|\big)^{\frac 1 2}.
\end{equation}

Let 
$K:\T^{d-1}\times\R\to [0,+\infty)$ be 
a measurable function satisfying
\begin{equation}\label{ass1K}
K(\xi)\le\frac{\gamma}{|\xi_d|^p}\qquad\textrm{for }|\xi_d|\ge 1
\end{equation}
and
\begin{equation}\label{ass2K}
K(\xi)\le \frac{\gamma}{\|\xi\|_{\sharp}^{d+s}}\qquad\textrm{for }|\xi_d|<1.
\end{equation}
Here, $p>2$, $0<s<1$, and $\gamma>0$. 

For every $E_1,E_2\in\Ins$ 
 we set
\begin{equation}\label{J}
\I_K(E_1,E_2):=\int_{E_1}\ud x\int_{E_2}K(x-y)\ud y,
\end{equation}
and we define the non-local perimeter governed by $K$ as
\begin{equation}\label{pernonloc}
\per_K(E):=\I_K(E,E^c)\qquad\textrm{for every }E\in\Ins.
\end{equation}
Here, $E^c=(\T^{d-1}\times \R)\setminus E$.
\begin{remark}
\rm{
Let $0<s<1$. 
Let $\psi\in L^\infty(\T^{d-1}\times\R;[0,+\infty))$ be a zero-homogeneous function.
Then, the kernel $K^\psi$ defined by
\begin{equation}\label{conpsi}
K^\psi(\xi):= \frac{1}{\|\xi\|^{d+s}_\sharp}\psi(\xi)\qquad\textrm{ for every }\xi\in\T^{d-1}\times\R
\end{equation}
satisfies \eqref{ass1K} and \eqref{ass2K}.

Notice that, for $\psi\equiv 1$,  the non-local perimeter \eqref{pernonloc} governed by $K^1$ represents a natural counterpart of the ``standard'' fractional perimeter \cite{CRS2010} for subsets of $\T^{d-1}\times\R$, where the euclidean distance in the 
interaction kernel
\begin{equation}\label{definJ}
J^s(\xi):= \frac{1}{|\xi|^{d+s}},
\end{equation}
is replaced by the geodesic one.
}
\end{remark}
\begin{proposition}\label{perfraz}
Let $K:\T^{d-1}\times\R\to [0,+\infty)$ be a measurable function satisfying  \eqref{ass1K} and \eqref{ass2K}. Then, the non-local perimeter $\per_K$ defined by \eqref{pernonloc} satisfies properties \eqref{iprop}-\eqref{viprop}.
\end{proposition}
\begin{proof}
One can easily check that properties \eqref{iprop}-\eqref{viprop} are satisfied.
Here we prove only \eqref{vprop}. 
To this end, let $E,F\in\Ins$. Then
\begin{equation*}
\begin{aligned}
\per_K(E\cup F)+\per_K(E\cap F)=&\,\I_K(E,E^c\cap F^c)+\I_K(E^c\cap F,E^c\cap F^c)\\
&\,+\I_K(E\cap F,F^c)+\I_K(E\cap F, F\cap E^c)\\
\le &\,\I_K(E,E^c\cap F^c)+\I_K(E, F\cap E^c)\\
&\,+\I_K(E^c\cap F,F^c)+\I_K(E\cap F,F^c)\\
=&\,\I_K(E,E^c)+\I_K(F,F^c)=\per_K(E)+\per_K(F).
\end{aligned}
\end{equation*}
\end{proof}
Now we introduce another possible extension to the periodic setting of the notion of fractional perimeter, accounting for the interaction of a set with the periodic copies of its complementary set in $\T^{d-1}\times\R$. 

To this end, for any given set $E\in \Ins_R$ we denote by $E_\sharp$ its ``periodic'' extension, i.e., $E_{\sharp}:=\bigcup_{z'\in\Z^{d-1}}(E+(z',0))$;
we define $\per_\sharp^s:\Ins_R\to (-\infty,+\infty]$ as
\begin{equation}\label{periper}
\per_\sharp^s(E):=\int_{\T^{d-1}\times\R}\ud x\int_{\R^d}\big(\chi_E(x)\chi_{E_\sharp^c}(y)-\chi_{H^-}(x)\chi_{H^+}(y)\big)J^s(x-y)\ud y,
\end{equation}
where $J^s$ is the kernel in \eqref{definJ}, and $H^-$ and $H^+$ are respectively the lower and upper halfspaces, that is 
$$H^-:=\{x=(x',x_d)\in\R^d\,:\,x_d<0\} \qquad \text{and} \qquad H^+:=\{x=(x',x_d)\in\R^d\,:\,x_d>0\}.$$ 
We first observe that the perimerer $\per_\sharp^s$ in \eqref{periper} is well-defined.

Indeed, by definition, we have
\begin{equation}\label{decompo}
\begin{aligned}
&\,\per_\sharp^s(E)\\
=&\,\int_{\S^d_{R}}\ud x\int_{\R^{d-1}\times(-R,R)}\big(\chi_E(x)\chi_{E_\sharp^c}(y)-\chi_{H^-}(x)\chi_{H^+}(y)\big)J^s(x-y)\ud y\\
&\,+\int_{\S^d_{R}}\ud x\int_{\R^d\setminus(\R^{d-1}\times (-R,R))}\big(\chi_E(x)\chi_{E_\sharp^c}(y)-\chi_{H^-}(x)\chi_{H^+}(y)\big)J^s(x-y)\ud y\\
&\,+\int_{(\T^{d-1}\times \R)\setminus\S^d_R}\ud x\int_{ \R^{d-1}\times (-R,R)}\big(\chi_E(x)\chi_{E_\sharp^c}(y)-\chi_{H^-}(x)\chi_{H^+}(y)\big)J^s(x-y)\ud y.
\end{aligned}
\end{equation}
Now, since
\begin{equation*}
\begin{aligned}
&\, \int_{\S^d_R}\ud x \int_{\R^{d-1}\times(-R,R)}\chi_{H^-}(x)\chi_{H^+}(y)J^s(x-y)\ud y\\
& \,+\int_{\S^{d}_R}\ud x\int_{\R^d\setminus (\R^{d-1}\times (-R,R))}J^s(x-y)\ud y
\\ 
& \,+\int_{(\T^{d-1}\times \R)\setminus\S^d_R}\ud x\int_{ \R^{d-1}\times (-R,R)}J^s(x-y)\ud y
\le C(R),
\end{aligned}
\end{equation*}
for some constant $C(R)>0$, by \eqref{decompo}, we get the well-posedness of the definition in \eqref{periper}.

\begin{remark}\label{rmkdeco}
Notice that, for any $E\in\Ins_R$,  $\per^s(E)$ can be rewritten as
\begin{equation}\label{decompo20}
\begin{aligned}
\per_\sharp^s(E)
=&\,\int_{\S^d_{R}}\ud x\int_{\R^{d-1}\times(-R,R)}\big(\chi_E(x)\chi_{E_\sharp^c}(y)-\chi_{H^-}(x)\chi_{H^+}(y)\big)J^s(x-y)\ud y\\
&\,+\,\int_{\S^d_{R}}\ud x\,\big(\chi_E(x)-\chi_{H^-}(x)\big)\\
&\,\phantom{\int_{\S^d_{R}}\ud x}\qquad\int_{\R^{d}}\big(\chi_{(R,+\infty)}(y_d)-\chi_{(-\infty,-R)}(y_d)\big)J^s(x-y)\ud y.
\end{aligned}
\end{equation}
To this end, we first notice that since $\chi_{E_\sharp^c}(y)=\chi_{H^+}(y)=0$ in $\T^{d-1}\times(-\infty,-R)$ and $\chi_{E_\sharp^c}(y)=\chi_{H^+}(y)=1$ in $\T^{d-1}\times(R,+\infty)$, 
\begin{equation}\label{decompo200}
\begin{aligned}
&\,\int_{\S^d_{R}}\ud x\int_{\R^d\setminus(\R^{d-1}\times (-R,R))}\big(\chi_E(x)\chi_{E_\sharp^c}(y)-\chi_{H^-}(x)\chi_{H^+}(y)\big)J^s(x-y)\ud y\\
=&\,\int_{\S^d_{R}}\ud x\int_{\R^{d-1}\times (R,+\infty)}\big(\chi_E(x)-\chi_{H^-}(x)\big)J^s(x-y)\ud y.
\end{aligned}
\end{equation}
Moreover,  using again that $\chi_E(x)=\chi_{H^-}(x)=0$ in $\T^{d-1}\times(R,+\infty)$ and $\chi_E(x)=\chi_{H^-}(x)=1$ in $\T^{d-1}\times(-\infty,-R)$, we get
\begin{equation}\label{seco}
\begin{aligned}
&\,\int_{(\T^{d-1}\times \R)\setminus\S^d_R}\ud x\int_{ \R^{d-1}\times (-R,R)}\big(\chi_E(x)\chi_{E_\sharp^c}(y)-\chi_{H^-}(x)\chi_{H^+}(y)\big)J^s(x-y)\ud y\\
=&\,\int_{\T^{d-1}\times(-\infty,-R)}\ud x\int_{\R^{d-1}\times (-R,R)}\big(\chi_{E_\sharp^c}(y)-\chi_{H^+}(y)\big)J^s(x-y)\ud y\\
=&\,-\sum_{z' \in \Z^{d-1}}\,\int_{\T^{d-1}\times(-\infty,-R)}\ud x\int_{(z'+\T^{d-1})\times (-R,R)}\big(\chi_{E_\sharp}(y)-\chi_{H^-}(y)\big)J^s(x-y)\ud y\\
=&\, -\sum_{z' \in \Z^{d-1}}\,\int_{(-z'+\T^{d-1})\times (-\infty,-R)}\ud \xi \int_{\T^{d-1}\times (-R,R)}\big(\chi_{E_\sharp}(\eta)-\chi_{H^-}(\eta)\big)J^s(\xi-\eta)\ud \eta\\
=&\,-\int_{\R^{d-1}\times (-\infty,-R)}\ud y\int_{\S^d_R}\big(\chi_{E}(x)-\chi_{H^-}(x)\big)J^s(x-y)\ud x,
\end{aligned}
\end{equation}
where we have used the change of variable $\xi=x-(z',0)$ and $\eta=y-(z',0)$  in the second equality   and the change of variable $y=\xi$ and $x=\eta$ in the last one.
Therefore, by \eqref{decompo}, \eqref{decompo200} and \eqref{seco}, we obtain
\begin{equation*}
\begin{aligned}
\per_\sharp^s(E)=&\,
\int_{\S^d_{R}}\ud x\int_{\R^{d-1}\times(-R,R)}\big(\chi_E(x)\chi_{E_\sharp^c}(y)-\chi_{H^-}(x)\chi_{H^+}(y)\big)J^s(x-y)\ud y\\
&\,+\,\int_{\S^d_{R}}\ud x\int_{\R^{d-1}\times (R,+\infty)}\big(\chi_E(x)-\chi_{H^-}(x)\big)J^s(x-y)\ud y\\
&\,-\int_{\S^d_R}\ud x \int_{\R^{d-1}\times (-\infty,-R)}\big(\chi_{E}(x)-\chi_{H^-}(x)\big)J^s(x-y)\ud y,
\end{aligned}
\end{equation*}
whence \eqref{decompo20} easily follows.
\end{remark}
Let us now consider sets which are not contained in any $\Ins_R$, for $R>0$. In these respects, for every $E \in \Ins$ we define (with a little abuse of notation) $\per_\sharp^s:\Ins \to (-\infty,+\infty]$ as
\begin{equation}\label{periodicogen}
\per_\sharp ^s (E):= \inf \left \{ \liminf_{k\to +\infty} \per_\sharp^s (E_k) \, : \, R_k>0, \,  \{E_k\}_{k \in \N} \subseteq \Ins_{R_k} , \, \chi_{E_k} \to \chi_E \, \text{in} \, L^1_{\loc} \right \}.
\end{equation}
It is easy to prove that the perimeter defined in \eqref{periodicogen} agrees with the definition in \eqref{periper} for sets in $\Ins_R$.  
Moreover, by arguing as above, one can show that $\per^s_\sharp$ is a generalized perimeter in sense of Definition \ref{per}.

\vskip5pt
{\bf Riesz-type perimeters.}
In addition to the fractional perimeter, our results also naturally apply to Riesz-type perimeters. In order to define them, we first introduce some notation.

For every $R_1, R_2 \in \R \cup \{\pm \infty\}$ with $R_1<R_2$, we set 
$$\S^d_{R_1,R_2}:= \T^{d-1} \times (R_1 , R_2) .$$
As a consequence, $\S_R^d=\S_{-R,R}^d$.

For every $E \subseteq \T^{d-1} \times \R$ with $|E|>0$, we set
\begin{equation}\label{roE}
\rho_E:= \inf \{ \rho\in\R \, : \, E \subseteq \S^d_{-\infty, \rho}\}  .
\end{equation} 
Let  $0 < \alpha < d-1$ and $R_1, R_2 \in \R \cup \{\pm \infty\}$ with $R_1<R_2$. For every $E \in \Ins$ we define
\begin{equation}\label{Per12}
\per^\alpha_{R_1,R_2} (E) := \iint_{\S^d_{R_1, R_2} \times \S^d_{R_1, R_2}} \frac{\ud x \ud y}{\|x-y\|_\sharp^{d-\alpha}}  - \iint_{(E \cap \S^d_{R_1,R_2}) \times ( E \cap \S^d_{R_1,R_2})} \frac{\ud x \ud y}{\|x-y\|_\sharp^{d-\alpha}}   .
\end{equation}
For every $E,F \in \Ins$ we set
\begin{equation}\label{Jalpha}
\J^\alpha(E,F): = \iint_{E\times F} \frac{\ud x \ud y}{\|x-y\|_\sharp^{d-\alpha}} .
\end{equation}
Using this notation, we have
\begin{equation*}
\per^\alpha_{R_1,R_2} (E)= \J^\alpha(\S^d_{R_1, R_2},\S^d_{R_1, R_2})-\J^\alpha(E \cap \S^d_{R_1,R_2},E \cap \S^d_{R_1,R_2})
\end{equation*}
and, equivalently,
\begin{equation}\label{defmonot}
\per^\alpha_{R_1,R_2} (E)= 2\J^\alpha(E \cap \S^d_{R_1, R_2}, E^c \cap \S^d_{R_1, R_2})+\J^\alpha(E^c \cap \S^d_{R_1,R_2},E^c \cap \S^d_{R_1,R_2})
\end{equation}
\begin{definition}\label{rieszper}
Let $0 < \alpha < d-1$. We define the $\alpha$--Riesz perimeter $\per^{\alpha}:\Ins\to [0,+\infty]$ as follows:
\begin{enumerate}
\item\label{tipo0} $\per^\alpha (\emptyset):=0$;
\item\label{tipo1} $\per^\alpha (E) := \lim_{R \to +\infty} \per_{-R,\rho_E}^\alpha (E)$,  if $\rho_E< + \infty$;
\item \label{tipo2} $\per^\alpha (E) :=\lim_{R \to +\infty} \per_{-R,R}^\alpha (E)$, if $\rho_E= + \infty$.
\end{enumerate}
\end{definition}
\begin{remark}\label{monot}
From expression \eqref{defmonot}, it easily follows that the perimeter $\per^\alpha_{R_1,R_2}$ is monotonically decreasing with respect to $R_1$ and monotonically increasing with respect to $R_2$. Hence, the definitions of $\per^\alpha$ in \eqref{tipo1} and \eqref{tipo2} are well-posed.
\end{remark}
\begin{remark}\label{perequiv}
We notice that,  if $\rho_E<+\infty$, then for every $R>|\rho_E|$ and for every $r\in[\rho_E,+\infty]$
$$\J^\alpha(E \cap \S^d_{-R, \rho_E},E \cap \S^d_{-R, \rho_E})=\J^\alpha(E \cap \S^d_{-R, r},E \cap \S^d_{-R, r}).
$$
\end{remark}
For every $E \in \Reg_L$, we denote by $f_E$
the $L$-Lipschitz continuous function such that 
\begin{equation}\label{deffE}
E=\{ (x',x_d) \in \T^{d-1}\times \R \, : \, x_d\leq f_E(x')\}
\end{equation}
and we set $i_E:= \min_{\T^{d-1}} f_E$.
\begin{proposition}\label{Palpha1}
For every $0<\alpha<d-1$, the perimeter $\per^\alpha$ satisfies property \eqref{iprop}.
\end{proposition}
\begin{proof} 
 Let $E \in \Reg_L$.
Fix $R>0$ such that $-R<i_E$. Recalling \eqref{defmonot}, we have
\begin{equation}\label{rieszfin}
 \begin{aligned}
 \per^\alpha_{-R,\rho_E} (E)=& 2\J^\alpha(E \cap \S^d_{-R,\rho_E}, E^c \cap \S^d_{-R,\rho_E})+\J^\alpha(E^c \cap \S^d_{-R,\rho_E},E^c \cap \S^d_{-R,\rho_E})
 \\=&2\J^\alpha(E \cap \S^d_{-R,i_E-1}, E^c \cap \S^d_{-R,\rho_E})+ 2\J^\alpha(E \cap \S^d_{i_E-1,\rho_E}, E^c \cap \S^d_{-R,\rho_E})
 \\&+\J^\alpha(E^c \cap \S^d_{-R,\rho_E},E^c \cap \S^d_{-R,\rho_E})\\
 =&2\J^\alpha(E \cap \S^d_{-R,i_E-1}, E^c \cap \S^d_{i_E,\rho_E})+ 2\J^\alpha(E \cap \S^d_{i_E-1,\rho_E}, E^c \cap \S^d_{i_E,\rho_E})
 \\&+\J^\alpha(E^c \cap \S^d_{i_E,\rho_E},E^c \cap \S^d_{i_E,\rho_E}).
 \end{aligned}
 \end{equation}
 Now we notice that the last two terms in the righthand side of \eqref{rieszfin} are finite since the kernel $\frac{1}{\|\cdot\|^{d-\alpha}_{\sharp}}$ is locally integrable for $0<\alpha<d-1$.
 Hence we need to estimate only the first interaction in the righthand side of \eqref{rieszfin} for which we find
 \begin{align*}
 \J^\alpha(E \cap \S^d_{-R,i_E}, E^c \cap \S^d_{i_E,\rho_E})
 =&\,  \int_{E \cap \S^d_{-R,i_E-1}}\int_{E^c \cap \S^d_{i_E,\rho_E}} \frac{\ud x \ud y}{\|x-y\|_\sharp^{d -\alpha}}
 \\
 \leq &\,\int_{\T^{d-1}} \ud x'  \int_{\T^{d-1}} \ud y' \int^{i_E-1}_{-R} \ud x_d \int_{i_E}^{\rho_E}  \frac{\ud y_d}{|x_d-y_d|^{d-\alpha}} ,
 \end{align*}
which is uniformly bounded (with respect to $R$) for $0<\alpha< d-1$.
The proof is concluded.
\end{proof}
\begin{proposition}\label{Palpha4}
For every $0<\alpha<d-1$, the perimeter $\per^\alpha$ satisfies the lower-semicontinuity property \eqref{ivprop}.
\end{proposition}
\begin{proof}
Let $\{E_n \}_{n\in\N}\subset \Ins$ and let $E\in\Ins$ be such that for every $R>0$
\begin{equation}\label{l1conve}
|(E_n\triangle E)\cap\S^{d}_R|\to 0\qquad\textrm{as }n\to +\infty.
\end{equation}
For the sake of clarity, we analyze all possible cases separately.
\vskip5pt
{\it Case a:} $\rho_{E_n}, \rho_E<+\infty$.

 Let $r < \rho_E$; by definition of $\rho_E$ we have $|\{(x',x_d) \in E \,: \; x_d>r\}|>0$. Then, by \eqref{l1conve} there exists $\overline n>0$ such that $|\{(x',x_d) \in E\cap E_n \,: \; x_d>r\} |>0$ for every $n>\overline n$. Hence, $\rho_{E_n}>r$ for every $n>\overline n$ and by the arbitrariness of $r$ it holds
\begin{equation}\label{semicontrho}
\liminf_{n \to \infty}\rho_{E_n} \geq \rho_E  .
\end{equation}
Thus, using \eqref{semicontrho} and Fatou Lemma in the first integral in \eqref{Per12} (with $R_2$ replaced by $\rho_{E_n}$ and $R_1$ replaced by $-R$) and
the Dominated Convergence Theorem in the second integral in \eqref{Per12}, taking into account  Remark \ref{perequiv}, for every $R>-\rho_E$ we get
\begin{align*}
\per_{-R,\rho_E}^\alpha (E)= &\,\J^\alpha ( \S^d_{-R, \rho_E}, \S^d_{-R, \rho_E})-\J^\alpha (E \cap \S^d_{-R, + \infty}, E \cap \S^d_{-R, + \infty}) 
    \\
    \leq&\, \liminf_{n \to +\infty}  \left[\J^\alpha ( \S^d_{-R, \rho_{E_n}}, \S^d_{-R, \rho_{E_n}})-\J^\alpha(E_n \cap \S^d_{-R, + \infty},E_n \cap \S^d_{-R, + \infty}) \right]
\\
\leq&\,  \liminf_{n \to +\infty} \per_{-R,\rho_{E_n}}^\alpha (E_n),
\end{align*}
hence, using also the monotonicity property in Remark \ref{monot}, we get 
\begin{equation}\label{lsc}
\begin{aligned}
 \per^\alpha (E) &= \lim_{R \to + \infty} \per_{-R,\rho_E}^\alpha (E)  
 \\& \leq   \limsup_{R \to + \infty}\liminf_{n \to +\infty} \per_{-R,\rho_{E_n}}^\alpha (E_n) 
 \\&  \leq  \limsup_{R \to + \infty} \liminf_{n \to +\infty} \per^\alpha (E_n) =  \liminf_{n \to +\infty} \per^\alpha (E_n).
\end{aligned}
\end{equation}
\vskip5pt
{\it Case b:} $\rho_{E_n}, \rho_E=+\infty$.

By the Dominated Convergence Theorem in \eqref{tipo2}, for every $R >0$ we get
$$ 
\per_{-R,R}^\alpha (E)= \lim_{n \to + \infty} \per_{-R,R}^\alpha (E_n) . 
$$
Letting $R \to + \infty$, we conclude arguing as in \eqref{lsc}.
\vskip5pt
{\it Case c:} $\rho_{E_n}=+\infty$ and $\rho_E< + \infty$.

By \eqref{l1conve}, using the Dominated Convergence Theorem, for every $R>|\rho_E|$ 
we have
$$
\per^\alpha_{-R, \rho_E}(E)= \lim_{n \to +\infty} \per^\alpha_{-R, \rho_E}(E_n) \leq \liminf_{n \to + \infty} \per_{-R,R}^\alpha (E_n)\le \liminf_{n \to + \infty} \per^\alpha (E_n),
$$
where in the last inequality we have used the monotonicity property in Remark \ref{monot}. Sending $R\to +\infty$, we have
$$
\per^\alpha(E)\le\liminf_{n\to +\infty}\per^\alpha(E_n).
$$
%
\vskip5pt
{\it Case d:} $\rho_{E_n}<+\infty$ for every $n\in\N$ and $\rho_E=+ \infty$. 

By arguing as in {\it Case a}, we have that 
\begin{equation}\label{rhodiv}
\rho_{E_n}\to +\infty\qquad\textrm{(as $n\to +\infty$)}.
\end{equation}
Fix $R>0$; by the Dominated Convergence Theorem,  using \eqref{rhodiv} and the monotonicity property in Remark \ref{monot}, we have
\begin{equation}\label{fina}
\per^\alpha_{-R,R}(E)=\lim_{n\to +\infty} \per^{\alpha}_{-R,R}(E_n)\le\liminf_{n\to+\infty}\per^{\alpha}_{-R,\rho_{E_n}}(E_n)\le\liminf_{n\to+\infty}\per^{\alpha}(E_n).
\end{equation}
By sending $R\to +\infty$ in \eqref{fina}, we get the claim also in this case. 

The proof is concluded.
\end{proof}
\begin{proposition}\label{Palpha5}
For every $0<\alpha<d-1$, the perimeter $\per^\alpha$ satisfies the submodularity property \eqref{vprop}.
\end{proposition}
\begin{proof} Let $E,F \in \Ins$. For every $R_1,R_2\in\R\cup\{\pm\infty\}$ with $R_1<R_2$, we set 
\begin{equation*}
U:= (E \setminus F) \cap \S_{R_1,R_2}^d,\qquad
V:=(F\setminus E) \cap \S_{R_1,R_2}^d,\qquad
W:= E\cap F \cap \S_{R_1,R_2}^d,
\end{equation*}
so that
\begin{equation*}
E \cap \S_{R_1,R_2}^d= U \dot{\cup} W, \quad
F \cap \S_{R_1,R_2}^d= V \dot{\cup}W,\quad
(E \cup F) \cap \S_{R_1,R_2}^d= U\dot{\cup} V \dot{\cup}W. 
\end{equation*}
Hence, recalling the definition of $\J^\alpha$ in \eqref{Jalpha}, we have
\begin{equation}\label{EuF}
\begin{aligned}
&\,\J^\alpha((E \cup F) \cap \S^d_{R_1,R_2} , (E \cup F) \cap \S_{R_1,R_2} ^d)\\
=&\,\J^\alpha(U ,U)+\J^\alpha(V,V)+\J^\alpha(W,W) +2\J^\alpha(U,V)+2 \J^\alpha(U,W)+2\J^\alpha(V,W) ,
\end{aligned}
\end{equation}
\begin{equation}\label{EnF}
\J^\alpha(E \cap F \cap \S_{R_1,R_2} ^d, E \cap F \cap \S_{R_1,R_2} ^d)=\J^\alpha(W ,W) ,
\end{equation}
\begin{equation}\label{E}
\J^\alpha(E \cap \S_{R_1,R_2} ^d, E \cap \S_{R_1,R_2} ^d)=\J^\alpha(U ,U)+ \J^\alpha(W,W)+2\J^\alpha(U ,W ) ,
\end{equation}
\begin{equation}\label{F}
\J^\alpha(F \cap \S_{R_1,R_2} ^d, F \cap \S_{R_1,R_2} ^d)=\J^\alpha(V,V )+ \J^\alpha(W,W)+2\J^\alpha(V ,W ) .
\end{equation}
Since $\J^{\alpha}(U,V)\ge 0$, by \eqref{EuF}-\eqref{F}, we have immediately that
\begin{equation}\label{unicaimp}
\begin{aligned}
&\,\J^\alpha((E \cup F) \cap \S^d_{R_1,R_2} , (E \cup F) \cap \S_{R_1,R_2} ^d)+\J^\alpha(E \cap F \cap \S^d_{R_1,R_2} , E \cap F \cap \S_{R_1,R_2} ^d)\\ 
=&\, \J^\alpha(E \cap \S_{R_1,R_2} ^d, E \cap \S_{R_1,R_2} ^d)+\J^\alpha(F\cap \S_{R_1,R_2} ^d, F \cap \S_{R_1,R_2} ^d)+\J^{\alpha}(U,V)\\
\ge&\,\J^\alpha(E \cap \S_{R_1,R_2} ^d, E \cap \S_{R_1,R_2} ^d)+\J^\alpha(F\cap \S_{R_1,R_2} ^d, F \cap \S_{R_1,R_2} ^d).
\end{aligned}
\end{equation}
We distinguish three cases according with the values $\rho_{E\cup F}$ and $\rho_{E\cap F}$.
\vskip5pt
{\it Case a:} $\rho_{E\cap F} =+ \infty$.

In such a case $\rho_E=\rho_F=\rho_{E\cup F}=+\infty$.
Let $R>0$.
By \eqref{Per12} and \eqref{unicaimp} we have
\begin{align*}
&\,\per^\alpha_{-R,R}(E\cup F)+ \per^\alpha_{-R,R}(E\cap F)\\
=&\,\J^\alpha(\S^d_R, \S^d_R)+\J^\alpha(\S^d_R, \S^d_R)\\
&\,-\J^\alpha((E \cup F) \cap \S^d_{R_1,R_2} , (E \cup F) \cap \S_{R_1,R_2} ^d)-\J^\alpha(E \cap F \cap \S^d_{R_1,R_2} , E \cap F \cap \S_{R_1,R_2} ^d)\\
\le&\,\per^\alpha_{-R,R}(E)+ \per^\alpha_{-R,R}( F),
\end{align*}
whence, letting $R \to+ \infty$, we obtain
\begin{equation}\label{Palphasubmod} 
\per^\alpha(E \cup F)+\per^\alpha(E \cap F) \leq \per^\alpha(E )+\per^\alpha(F) .
\end{equation}
\vskip5pt
{\it Case b:} $\rho_{E\cup F} < +\infty$. 

In such a case, $\rho_E$ and $\rho_F$ are both finite; we can assume without loss of generality that $\rho_{E}\le\rho_F$, so that
 $\rho_{E \cup F}= \rho_F$ and $\rho_{E \cap F}\leq \rho_E$.
 Let $R>|\rho_F|$. 
Hence, by Remark \ref{monot}, \eqref{Per12} and \eqref{unicaimp}, we get
\begin{align*}
\per^\alpha_{-R, \rho_{E \cup F}}(E\cup F)+ \per^\alpha_{-R, \rho_{E \cap F}}(E\cap F)
\le&\,\per^\alpha_{-R, \rho_{F}}(E\cup F)+ \per^\alpha_{-R, \rho_{E}}(E\cap F)\\
=&\, \J^\alpha(\S^d_{-R, \rho_{F}}, \S^d_{-R, \rho_{F}})+  \J^\alpha(\S^d_{-R, \rho_{E}}, \S^d_{-R, \rho_{E}})\\
&\,- \J^\alpha((E\cup F) \cap \S^d_{-R, \rho_{F}}, (E\cup F) \cap \S^d_{-R, \rho_{F}})
\\&\,-\J^\alpha(E\cap F \cap \S^d_{-R, \rho_{E}}, E\cap F \cap \S^d_{-R, \rho_{E}})\\
\le&\,\per^\alpha_{-R, \rho_{E}}(E)+ \per^\alpha_{-R, \rho_{F}}( F),
\end{align*}
whence, sending $R\to +\infty$, \eqref{Palphasubmod} follows also in this case. 
\vskip5pt
{\it Case c:} $\rho_{E\cup F}= +\infty$ and $ \rho_{E\cap F}< +\infty$. 

We can assume without loss of generality that $\rho_E\le\rho_F$ so that $\rho_F=+\infty$.

Now, if $\rho_E=+\infty$, then, using Remark \ref{monot},  \eqref{Per12} and \eqref{unicaimp}, for every $R>|\rho_{E\cap F}|$ we have  
\begin{align*}
\per^\alpha_{-R, R}(E\cup F)+ \per^\alpha_{-R, \rho_{E\cap F}}(E\cap F)
\le&\,\per^\alpha_{-R, R}(E\cup F)+ \per^\alpha_{-R, R}(E\cap F)\\
=&\, \J^\alpha(\S^d_{R}, \S^d_{R})+  \J^\alpha(\S^d_{R}, \S^d_{R})\\
&\,- \J^\alpha((E\cup F) \cap \S^d_{R}, (E\cup F) \cap \S^d_{R})
\\&\,-\J^\alpha(E\cap F \cap \S^d_{R}, E\cap F \cap \S^d_{R})\\
\le&\,\per^\alpha_{-R, R}(E)+ \per^\alpha_{-R, R}( F);
\end{align*}
by sending $R\to +\infty$, we obtain \eqref{Palphasubmod}. 

If $\rho_E<+\infty$, then $\rho_{E\cap F}\le\rho_E$, so that using Remark \ref{monot},  \eqref{Per12}, Remark \ref{perequiv} and \eqref{unicaimp}, for every $R>|\rho_{E}|$ we have
\begin{align*}
&\,\per^\alpha_{-R, R}(E\cup F)+ \per^\alpha_{-R, \rho_{E \cap F}}(E\cap F)\\
\le&\,\per^\alpha_{-R, R}(E\cup F)+ \per^\alpha_{-R, \rho_{E}}(E\cap F)\\
=&\, \J^\alpha(\S^d_{R}, \S^d_{R})+  \J^\alpha(\S^d_{-R, \rho_{E}}, \S^d_{-R, \rho_{E}})\\
&\,- \J^\alpha((E\cup F) \cap \S^d_{R}, (E\cup F) \cap \S^d_{R})-\J^\alpha(E\cap F \cap \S^d_{-R, \rho_{E}}, E\cap F \cap \S^d_{-R, \rho_{E}})\\
=&\, \J^\alpha(\S^d_{R}, \S^d_{R})+  \J^\alpha(\S^d_{-R, \rho_{E}}, \S^d_{-R, \rho_{E}})\\
&\,- \J^\alpha((E\cup F) \cap \S^d_{R}, (E\cup F) \cap \S^d_{R})-\J^\alpha(E\cap F \cap \S^d_{R}, E\cap F \cap \S^d_{R})\\
\le&\,\J^\alpha(\S^d_{R}, \S^d_{R})+  \J^\alpha(\S^d_{-R, \rho_{E}}, \S^d_{-R, \rho_{E}})- \J^\alpha(F \cap \S^d_{R}, F \cap \S^d_{R})-\J^\alpha(E \cap \S^d_{R}, E \cap \S^d_{R})\\
=&\J^\alpha(\S^d_{R}, \S^d_{R})+  \J^\alpha(\S^d_{-R, \rho_{E}}, \S^d_{-R, \rho_{E}})- \J^\alpha(F \cap \S^d_{R}, F \cap \S^d_{R})-\J^\alpha(E \cap \S^d_{-R,\rho_E}, E \cap \S^d_{-R,\rho_E})\\
=&\,\per^\alpha_{-R, \rho_E}(E)+ \per^\alpha_{-R, R}( F);
\end{align*}
by sending $R\to +\infty$, we obtain \eqref{Palphasubmod} also in this case.

This concludes the proof of the whole proposition. 
\end{proof}
\begin{proposition}\label{Palpha6}
For every $0<\alpha<d-1$, the perimeter $\per^\alpha$ satisfies the translational invariance property \eqref{viprop}. 
\end{proposition}
\begin{proof}
Let $E \in \Ins$ and $\tau\in\R^d$.
\vskip5pt
{\it Case a:} $\rho_E=+ \infty$.

Let $R>0$. 
Since $(E+\tau)\cap \S^d_R=E\cap\S^{d}_{-R-\tau_d,R-\tau_d}$, using Remark \eqref{monot}, 
we deduce 
 that $\per^\alpha_{-R,R} (E+\tau)=\per^\alpha_{-R-\tau_d,R-\tau_d}(E)$, so that
 $$
\per^\alpha_{-R+|\tau_d|,R-|\tau_d|}(E)\le \per^\alpha_{-R,R} (E+\tau)\le \per^\alpha_{-R-|\tau_d|,R+|\tau_d|}(E).
 $$
By sending $R\to +\infty$, we get 
\begin{equation}\label{trinv}
\per^\alpha(E+\tau)=\per^\alpha(E).
\end{equation}
\vskip5pt
{\it Case b:} $\rho_E<+\infty$.

By the very definition of $\rho_E$ in \eqref{roE} we have $\rho_{E+\tau}=\rho_E+\tau_d$. 
Hence, by arguing as in {\it Case a}, we have that $\per^\alpha_{-R,\rho_{E+\tau}} (E+\tau)=\per^\alpha_{-R-\tau_d,\rho_E}(E)$ (for every $R>|\rho_{E}+\tau_d|$), whence, sending $R\to +\infty$ we obtain \eqref{trinv} also in this case.
\end{proof}
\begin{proposition}\label{perriesz}
For every $0<\alpha<d-1$, the perimeter $\per^\alpha$ in Definition \ref{rieszper} is a generalized perimeter in the sense of Definition \ref{per}.
\end{proposition}
\begin{proof}
Properties \eqref{iprop}, \eqref{ivprop}, \eqref{vprop} and \eqref{viprop} follow from Propositions \ref{Palpha1}, \ref{Palpha4}, \ref{Palpha5} and \ref{Palpha6}, respectively. Property \eqref{iiiprop} is trivially satisfied. Finally, \eqref{iiprop} is an immediate  consequence of \eqref{tipo0} and of the fact that $\T^{d-1}\times \R$ is a set of type \eqref{tipo2} satisfying $\per^\alpha_{-R,R}(\T^{d-1}\times \R)=0$ for every $R>0$.
\end{proof}

\vskip5pt
{\bf The $0$-fractional perimeter.}
Here, we show how our results apply also to a suitable variant of the $0$-fractional perimeter, introduced in \cite{DNP} (see also \cite{CW}) for sets with finite measure.
To this end,
we set 
$$
D_{<}:=\{(x,y)\in\R^{d}\times\R^d\,:\,|x-y|<1\} \textrm{ and } D_{>}:=\{(x,y)\in\R^{d}\times\R^d\,:\,|x-y|>1\}.
$$
and, for every $E,F\in\Ins$ we define
\begin{equation*}
\I^{0}(E,F):=\iint_{(E\times F)\cap D_{<}}\frac{1}{\|x-y\|_{\sharp}^{d}}\ud y\ud x\textrm{ and }\J^{0}(E,F):=\iint_{(E\times F)\cap D_{>}}\frac{1}{\|x-y\|_{\sharp}^{d}}\ud y\ud x.
\end{equation*}
Moreover, for every $0<R_1<R_2$, and for every $E\in\Ins$, we define:
$$
\Q^0_{R_1,R_2}(E):=
\J^0(\S^d_{R_1,R_2},\S^d_{R_1,R_2})-\J^0(E\cap\S^d_{R_1,R_2},E\cap\S^d_{R_1,R_2}).
$$
As in Remark \ref{monot}, the functional $\Q^0_{R_1,R_2}$ is monotonically non-increasing with respect to $R_1$ and monotonically non-decreasing with respect to $R_2$.

We recall the definition of $\rho_E$ in \eqref{roE}. 
\begin{definition}\label{per0}
The $0$-perimeter $\per^0:\Ins\to [0,+\infty]$ is defined as follows:
\begin{itemize}
\item[(1)] $\per^0(\emptyset)=0$;
\item[(2)] $\per^0(E):=\I^0(E,E^c)+\lim_{R\to +\infty}\Q^0_{-R,\rho_E}(E)$ if $\rho_E<+\infty$;
\item[(3)] $\per^0(E):=\I^0(E,E^c)+\lim_{R\to +\infty}\Q^0_{-R,R}(E)$ if $\rho_E=+\infty$.
\end{itemize}
\end{definition}
\begin{proposition}
The $0$-fractional perimeter in Definition \ref{per0} is a generalized perimeter in the sense of Definition \ref{per}.
\end{proposition}
\begin{proof}
By applying Proposition \ref{perfraz} with $K(\cdot):=\frac{\chi_{B_1}(\cdot)}{\|\cdot\|^d}$ we have that $\I^0(\cdot,\cdot^c)$ is a generalized perimeter. 
Furthermore, by arguing verbatim as in the proof of Proposition \ref{perriesz} (and hence as in the proofs of Propositions  \ref{Palpha1}-\ref{Palpha6}), replacing the kernel $\frac{1}{\|\cdot\|_\sharp^{d-\alpha}}$ with $\frac{\chi_{B_1^{c}}(\cdot)}{\|\cdot\|_\sharp^{d}}$, we have that also the other term in the definition of $\per^0$ defines a generalized perimeter.
Since the sum of generalized perimeters is still a generalized perimeter we get the claim.
\end{proof}
\vskip5pt
{\bf The Minkowski pre-content.} 
Let us conclude this section by introducing a last example. Given $\rho>0$, for every set $E \in \Ins$ we define the Minkowski pre-content of $E$ as
\begin{equation}\label{mink2}
\mathcal E_\rho (E) := \frac{1}{2 \rho} \int_{\T^{d-1} \times \R} \osc_{B_\rho(x)}\chi_E \ud x,
\end{equation}
where, for every measurable function $u$ and for every measurable set $A$, $\osc_A u := {\mathrm{ess}} \sup_A u-{\mathrm{ess}} \inf_A u$ denotes the essential oscillation $u$ over $A$.
%

As highlighted in \cite{CMP12} (see also \cite{CMP}), the perimeter defined in \eqref{mink2} is a {\em generalized perimeter} in sense of Definition \ref{per}.

\section{Convergence to the halfspace}\label{convergencesec}
Given 
a set $E \in \Reg_L$, for every $0<\ep<1$ we define $E_\ep:=\{ (x',x_d) \in \T^{d-1}\times \R \, : \, x_d\leq(1-\ep) f_E(x')\}$, where $f_E$ is the $L$-Lipschitz continuous function in \eqref{deffE}. We assume that:
\begin{itemize}
\item[(\rm{H})] For every $\delta>0$, there exists $0<\ep_0<1$ such that, if 
\begin{equation}\label{oscgr}
\osc_{\T^{d-1}}f_E\geq \delta,
\end{equation} 
then
\begin{equation}
\label{oranu}
\per (E)-\per(E_\ep) \geq C(\delta, L) \ep\qquad\textrm{for every }\ep\le\ep_0,
\end{equation}
for some constant $C(\delta,L)>0$.
\end{itemize}
\begin{proposition}\label{quantiver}
Let $\per$ be a generalized perimeter in the sense of Definition \ref{per} satisfying property {\rm (H)}.
Let $E_0\in\Ins_R\cap\Reg_L$ with $\per(E_0)<+\infty$, and  $E_t^0$ the flat flow associated to $\per$ with initial datum $E_0$. Then, $f_{E_t^0}$ uniformly converges - as $t\to +\infty$ - to some constant $\lambda\in[-R,R]$.
\end{proposition}
\begin{proof}
Let $\delta>0$ be fixed and let $\ep_0$ be the constant provided by assumption (H).
We preliminarily observe that, for every $h>0$, for every $0<\ep<1$ and for every $E\in\Ins_R\cap\Reg_L$
\begin{equation}\label{comevadiss}
\frac 1h \int_{E \Delta E_\ep} \dist(x, \partial E) \ud x  \le C(L) \frac{\ep^2}{h}.
\end{equation}
Let $h\le\ep_0$ and  let $\ep=\min\{\frac{C(\delta,L)}{2C(L)},1\} h$, where $C(\delta,L)$ and $C(L)$ are the constants in \eqref{oranu} and \eqref{comevadiss}, respectively. 
By assumption (H) and by \eqref{comevadiss} we get that
for every $k\in\N$ such that $\osc_{\T^{d-1}}f_{E_{kh}^{(h)}}\ge \delta$ it holds
\begin{equation}\label{assconv}
\per (E_{kh}^{(h)})-\left[ \per((E_{kh}^{(h)})_\ep)+ \frac 1h \int_{E_{kh}^{(h)} \Delta (E_{kh}^{(h)})_\ep} \dist(x, \partial E_{kh}^{(h)}) \ud x \right] \geq C'(\delta, L) h,
\end{equation}
which,
by the minimality of $E^{(h)}_{(k+1)h}$, implies
\begin{equation}\label{stimaEk}
\per (\Eh_{kh})-\left[ \per(\Eh_{(k+1)h})+ \frac 1h \int_{\Eh _{kh}\Delta \Eh_{(k+1)h}} \dist(x, \partial \Eh_{kh}) \ud x \right] \geq C'(\delta, L) h.
\end{equation}
As a consequence, for every $k\in\N$ such that $\osc_{\T^{d-1}}f_{E_{kh}^{(h)}}\ge\delta$ we have
\begin{equation}\label{stimaEK1}
\per (\Eh_{kh})- \per(\Eh_{(k+1)h})\geq C'(\delta, L) h.
\end{equation}
Let $N_\delta\in\N$ be such that 
\begin{equation}\label{propNdelta}
\osc_{\T^{d-1}} f_{\Eh_{k h}}\ge\delta\qquad\textrm{ for every }k=0,1,\ldots,N_\delta.
\end{equation}
Then, by \eqref{stimaEK1},
%
\begin{equation*}
 \per (\Eh_0)-\per(\Eh_{N_\delta h})=\sum_{k=0}^{N_\delta} \per (\Eh_{kh})-\per(\Eh_{(k+1)h}) \geq C'(\delta,L) N_\delta h
\end{equation*}
that is, $h N_\delta \leq C(\delta,L,\per(E_0))$. Let $N_\delta^{\max}\in \N$ be the maximal $N_\delta$ for which \eqref{propNdelta} holds. Clearly, 
$\osc_{\T^{d-1}}f_{\Eh_{(N_\delta^{\max}+1)h}}\le \delta$ and, by Lemma \ref{lemmamin}, $\osc_{\T^{d-1}}f_{\Eh_{kh}}\le \delta$ for every $k\ge N_\delta^{\max}+1$. 
Therefore, $\osc_{\T^{d-1}}f_{\Eh_{t}}\le \delta$ for every $t\ge C(\delta,L,\per(E_0))$ and hence, taking the limit as $h\to 0$, the same holds true for the flat flow $E^0_t$.
It follows that 
\begin{equation}\label{osclim}
\osc_{\T^{d-1}}f_{E^0_{t}}\to 0\qquad\textrm{ as }t\to +\infty. 
\end{equation}
Let $\xi\in\T^{d-1}$ be fixed. Then, there exists a sequence $\{t_n\}_{n\in\N}$ with $t_n\to +\infty$, such that $f_{E^0_{t_n}}(\xi)\to c$ (for some $c\in[-R,R]$) as $n\to +\infty$ and, by \eqref{osclim}, 
\begin{equation}\label{unicon}
f_{E^0_{t_n}}\to c\qquad\textrm{uniformly in }\T^{d-1}.
\end{equation}
Now, by Lemma \ref{lemmamin}, sending $h\to 0$, we have that $\osc_{\T^{d-1}}f_{E^0_t}\le \osc_{\T^{d-1}}f_{E^0_{t'}}$ for every $t\ge t'$, which, together with \eqref{unicon}, implies that $f_{E^0_t}\to c$ - as $t\to +\infty$ - uniformly in $\T^{d-1}$.


\end{proof}
The remaining part of this section is devoted to exhibit examples of generalized perimeters  satisfying assumption (H).
\begin{proposition}
The Euclidean perimeter  $\per^{\mathrm{Eu}}$ satisfies assumption {\rm (H)}.
\end{proposition}
\begin{proof}
Let $E\in\Reg_L$ satisfy \eqref{oscgr}.
Then, for every $0<\ep<1$
\begin{equation*}
\begin{aligned}
\per^{\mathrm{Eu}}(E)-\per^{\mathrm{Eu}}(E_\ep) =&\, \int_{\T^{d-1} } \sqrt{1+|\nabla f_E|^2} \ud x' - \int_{\T^{d-1} } \sqrt{1+(1-\ep)^2|\nabla f_E|^2} \ud x'  
\\=&\, \int_{\T^{d-1} }\frac{ (1+|\nabla f_E|^2) -(1+(1-\ep)^2|\nabla f_E|^2) }{\sqrt{1+|\nabla f_E|^2}+ \sqrt{1+(1-\ep)^2|\nabla f_E|^2}}\ud x'   
\\=&\, \int_{\T^{d-1} }\frac{ \ep (2-\ep)|\nabla f_E|^2}{\sqrt{1+|\nabla f_E|^2}+ \sqrt{1+(1-\ep)^2|\nabla f_E|^2}}\ud x'  
\\
\ge&\, \frac{\ep}{2\sqrt{1+L^2}}\int_{\T^{d-1} }|\nabla f_E|^2\ud x'\geq C(L,\delta) \ep,
\end{aligned}
\end{equation*}
where the last inequality follows by \eqref{oscgr}.
\end{proof}
\begin{proposition}\label{propositionfraz}
The fractional perimeter $\per_K$ defined in \eqref{pernonloc} satisfies assumption {\rm (H)} for every kernel $K: \T^{d-1} \times \R \to [0,+\infty)$ satisfying \eqref{ass1K}, \eqref{ass2K} and such that $K(x',x_d)=K(-x',x_d)=K(x',-x_d)$.
\end{proposition}
\begin{proof}
Let $E\in\Reg_L$ satisfy \eqref{oscgr}. We can assume without loss of generality that $f_E>0$.
Then, by a change of variable,
\begin{equation}\label{cozero}
\begin{aligned}
&\,\per_K(E)-\per_{K}(E_\ep)\\
=&\,\int_{\T^{d-1}}\ud x'\int_{\T^{d-1}}\ud y'\int_{-\infty}^{f_E(x')}\ud x_d\int_{f_E(y')}^{+\infty}K(x-y)\ud y_d\\
&\,-\int_{\T^{d-1}}\ud x'\int_{\T^{d-1}}\ud y'\int_{-\infty}^{(1-\ep)f_E(x')}\ud x_d\int_{(1-\ep)f_E(y')}^{+\infty}K(x-y)\ud y_d\\
=&\,\int_{\T^{d-1}}\ud x'\int_{\T^{d-1}}\ud y'\int_{(1-\ep)f_E(x')}^{f_E(x')}\ud x_d\int_{f_E(y')}^{+\infty}K(x-y)\ud y_d\\
&\,-\int_{\T^{d-1}}\ud x'\int_{\T^{d-1}}\ud y'\int_{-\infty}^{(1-\ep)f_E(x')}\ud x_d\int_{(1-\ep)f_E(y')}^{f_E(y')}K(x-y)\ud y_d\\
=&\,\int_{\T^{d-1}}\ud x'\int_{\T^{d-1}}\ud y'\int_{(1-\ep)f_E(x')}^{f_E(x')}\ud x_d\int_{f_E(y')}^{+\infty}K(x-y)\ud y_d\\
&
-\int_{\T^{d-1}}\ud x'\int_{\T^{d-1}}\ud y'\int_{(1-\ep)f_E(x')}^{f_E(x')}\ud x_d
\int_{-\infty}^{(1-\ep)f_E(y')}K(x-y)\ud y_d.
\end{aligned}
\end{equation}
By using the change of variable
\begin{equation}\label{changevar}
\hat x_d= -x_d+(2-\ep)f_E(x'),\qquad\qquad \hat y_d=-y_d+(2-\ep)f_E(x'),
\end{equation} 
from \eqref{cozero}, we get
\begin{equation}\label{interm1}
\begin{aligned}
&\,\per_K(E)-\per_{K}(E_\ep)\\
=&\,\int_{\T^{d-1}}\ud x'\int_{\T^{d-1}}\ud y'\int_{(1-\ep)f_E(x')}^{f_E(x')}\ud x_d\int_{f_E(y')}^{+\infty}K(x-y)\ud y_d\\
&\,-\int_{\T^{d-1}}\ud x'\int_{\T^{d-1}}\ud y'\int_{(1-\ep)f_E(x')}^{f_E(x')}\ud \hat x_d\\
&\,\phantom{\int_{\T^{d-1}}\ud x'\int_{\T^{d-1}}\ud y'}\int_{f_E(y')+(2-\ep)(f_E(x')-f_E(y'))}^{+\infty}K(x'-y', \hat x_d-\hat y_d)\ud \hat y_d\\
=&\,\iint_{A^>}\ud x'\ud y'\int_{(1-\ep)f_E(x')}^{f_E(x')}\ud x_d\int_{f_E(y')}^{f_E(y')+(2-\ep)(f_E(x')-f_E(y'))}K(x-y)\ud y_d\\
&\,-\iint_{A^<}\ud x'\ud y'\int_{(1-\ep)f_E(x')}^{f_E(x')}\ud x_d\int_{f_E(y')-(2-\ep)(f_E(y')-f_E(x'))}^{f_E(y')}K(x-y)\ud y_d,
\end{aligned}
\end{equation}
where we have set 
\begin{equation}\label{insiemi}
\begin{aligned}
A^>:=&\,\{(x',y')\in \T^{d-1}\times \T^{d-1}\,:\,f_E(x')>f_E(y')\}\\
A^{<}:=&\,\{(x',y')\in \T^{d-1}\times \T^{d-1}\,:\,f_E(x')< f_E(y')\}.
\end{aligned}
\end{equation}
For every $\ep>0$ we set $\D(\ep):=\per_K(E)-\per_{K}(E_\ep)$.
In order to prove the claim, it is enough to show that
\begin{equation}\label{stimader}
\D'(0)\ge C(\delta,L),
\end{equation}
for some positive constant $C(\delta,L)$.
Now we prove \eqref{stimader}. By the very definition of $\D$, in view of \eqref{interm1}, we have
\begin{equation}\label{quasiconclu}
\begin{aligned}
\D'(0)=&\,\iint_{A^>}\ud x'\ud y' f_E(x')\int_{f_E(y')}^{2f_E(x')-f_E(y')}K(x'-y',f_E(x')-y_d)\ud y_d\\
&\,-\iint_{A^<}\ud x'\ud y'f_E(x')\int_{2f_E(x')-f_E(y')}^{f_E(y')}K(x'-y',f_E(x')-y_d)\ud y_d\\
=&\,\iint_{A^>}\ud x'\ud y' f_E(x')\int_{f_E(y')}^{2f_E(x')-f_E(y')}K(x'-y',f_E(x')-y_d)\ud y_d\\
&\,-\iint_{A^<}\ud x'\ud y'f_E(y')\int_{2f_E(x')-f_E(y')}^{f_E(y')}K(x'-y',f_E(x')-y_d)\ud y_d\\
&\,+\iint_{A^<}\ud x'\ud y'(f_E(y')-f_E(x'))\int_{2f_E(x')-f_E(y')}^{f_E(y')}K(x'-y',f_E(x')-y_d)\ud y_d\\
=&\iint_{A^>}\ud x'\ud y' f_E(x')\Big(\int_{f_E(y')}^{2f_E(x')-f_E(y')}K(x'-y',f_E(x')-y_d)\ud y_d\\
&\,\phantom{\iint_{A^>}\ud x'\ud y' f_E(x')}-\int_{2f_E(y')-f_E(x')}^{f_E(x')}K(x'-y',f_E(y')-y_d)\ud y_d\Big)\\
&\,+\iint_{A^<}\ud x'\ud y'(f_E(y')-f_E(x'))\int_{2f_E(x')-f_E(y')}^{f_E(y')}K(x'-y',f_E(x')-y_d)\ud y_d.
\end{aligned}
\end{equation}
Notice that for every $(x',y')\in A^>$
\begin{equation}\label{valezero}
\begin{aligned}
&\,\int_{f_E(y')}^{2f_E(x')-f_E(y')}K(x'-y',f_E(x')-y_d)\ud y_d\\
&\,\phantom{\int_{f_E(y')}^{2f_E(x')-f_E(y')}}-\int_{2f_E(y')-f_E(x')}^{f_E(x')}K(x'-y',f_E(y')-y_d)\ud y_d\\
=&\,\int_{f_E(y')-f_E(x')}^{f_E(x')-f_E(y')}K(x'-y',\eta)\ud \eta-\int_{f_E(y')-f_E(x')}^{f_E(x')-f_E(y')}K(x'-y',\eta)\ud \eta=0,
\end{aligned}
\end{equation}
where we have used the change of variable $\eta=y_d-f_E(x')$ in the first integral and 
$\eta=y_d-f_E(y')$ in the second one.
Now, setting 
\begin{equation}\label{insiemedelta}
A^{<}_\delta:=\Big\{(x',y')\in A^{<}\,:\,f_E(y')-f_E(x')>\frac\delta 2\Big\},
\end{equation}
since $f_E$ is $L$-Lipschitz continuous and in view of \eqref{oscgr}, we have that 
\begin{equation}\label{mispo}
|A^{<}_\delta|>C(\delta,L),
\end{equation}
for some positive constant $C(\delta,L)$.
Therefore, by \eqref{quasiconclu} and \eqref{valezero}, using the change of variable $\eta=y_d-f_E(x')$, we deduce
\begin{equation*}
\begin{aligned}
\D'(0)=&\,\iint_{A^<}\ud x'\ud y'(f_E(y')-f_E(x'))\int_{2f_E(x')-f_E(y')}^{f_E(y')}K(x'-y',f_E(x')-y_d)\ud y_d\\
\ge&\,\frac{\delta}{2}\iint_{A_\delta^<}\ud x'\ud y\int_{2f_E(x')-f_E(y')}^{f_E(y')}K(x'-y',f_E(x')-y_d)\ud y_d\\
=&\,\frac{\delta}{2}\iint_{A_\delta^<}\ud x'\ud y\int_{f_E(x')-f_E(y')}^{f_E(y')-f_E(x')}K(x'-y',\eta)\ud \eta\\
\ge&\,\frac{\delta}{2}\iint_{A_\delta^<}\ud x'\ud y\int_{-\frac{\delta}2}^{\frac{\delta}{2}}K(x'-y',\eta)\ud \eta
\ge C(\delta,L),
\end{aligned}
\end{equation*}
where the last inequality follows by \eqref{mispo}. This concludes the proof.
\end{proof}
\begin{proposition}
For every $0<s<1$ the fractional perimeter $\per^s_\sharp$ defined in \eqref{periper} satisfies assumption {\rm (H)}.
\end{proposition}
\begin{proof}
Let $E\in\Reg_L\cap\Ins_R$.
Then, taking into account \eqref{decompo20}, we have
\begin{equation}\label{teco}
\begin{aligned}
&\,\per_\sharp^s(E)-\per_\sharp^s(E_\ep)\\
=  &  \, \int_{\S^d_{R}}\ud x\int_{\R^{d-1}\times (-R,R)}\big(\chi_E(x)\chi_{E_\sharp^c}(y)-\chi_{E_\ep}(x)\chi_{E^c_{\ep, \sharp}}(y)\big)J^s(x-y)\ud y   \\
&\,+\,\int_{\S^d_{R}}\ud x\,\big(\chi_E(x)-\chi_{E_\ep}(x)\big)\\
&\,\phantom{\int_{\S^d_{R}}\ud x}\qquad\int_{\R^{d}}\big(\chi_{(R,+\infty)}(y_d)-\chi_{(-\infty,-R)}(y_d)\big)J^s(x-y)\ud y.
\end{aligned}
\end{equation}
Assuming without loss of generality that $f_E>0$, 
we have
\begin{equation}
\begin{aligned}
\per_\sharp^s(E)-\per_\sharp^s(E_\ep)=&
\,\int_{\T^{d-1}}\ud x'\int_{\R^{d-1}}\ud y'\int_{-R}^{f_E(x')} \ud x_d \int_{f_E(y')}^R J^s(x-y)\ud y_d  \\
& \,- \int_{\T^{d-1}}\ud x'\int_{\R^{d-1}}\ud y'\int_{-R}^{(1-\ep)f_E(x')} \ud x_d \int_{(1-\ep)f_E(y')}^{R} J^s(x-y)\ud y_d  \\
 &\,+ \int_{\T^{d-1}}\ud x'\int_{\R^{d-1}}\ud y'\int_{(1-\ep)f_E(x')}^{f_E(x')} \ud x_d\\
&\,\phantom{\int_{\S^d_{R}}\ud x}\qquad\int_{\R^{d}}\big(\chi_{(R,+\infty)}(y_d)-\chi_{(-\infty,-R)}(y_d)\big)J^s(x-y)\ud y,
\end{aligned}
\end{equation}
so that
\begin{equation}\label{tersomma}
\begin{aligned}
\per_\sharp^s(E)-\per_\sharp^s(E_\ep)=&
 \,\int_{\T^{d-1}}\ud x'\int_{\R^{d-1}}\ud y'\int_{(1-\ep)f_E(x')}^{f_E(x')} \ud x_d \int_{f_E(y')}^R J^s(x-y)\ud y_d  \\
& \,- \int_{\T^{d-1}}\ud x'\int_{\R^{d-1}}\ud y'\int_{-R}^{(1-\ep)f_E(x')} \ud x_d \int_{(1-\ep)f_E(y')}^{f_E(y')} J^s(x-y)\ud y_d  \\
&\, +\, \int_{\T^{d-1}}\ud x'\int_{\R^{d-1}}\ud y'\int_{(1-\ep)f_E(x')}^{f_E(x')} \ud x_d\\
&\,\phantom{\int_{\S^d_{R}}\ud x}\qquad\int_{\R}\big(\chi_{(R,+\infty)}(y_d)-\chi_{(-\infty,-R)}(y_d)\big)J^s(x-y)\ud y_d.
\end{aligned}
\end{equation}
Notice that, by periodicity,
\begin{equation}\label{perio}
\begin{aligned}
&\,\int_{\T^{d-1}}\ud x'\int_{\R^{d-1}}\ud y'\int_{-R}^{(1-\ep)f_E(x')} \ud x_d \int_{(1-\ep)f_E(y')}^{f_E(y')} J^s(x-y)\ud y_d \\
=&\,\int_{\R^{d-1}}\ud x'\int_{\T^{d-1}}\ud y'\int_{-R}^{(1-\ep)f_E(x')} \ud x_d \int_{(1-\ep)f_E(y')}^{f_E(y')} J^s(x-y)\ud y_d \\
=&\,\int_{\T^{d-1}}\ud x'\int_{\R^{d-1}}\ud y' \int_{(1-\ep)f_E(x')}^{f_E(x')}\ud x_d \int_{-R}^{(1-\ep)f_E(y')} J^s(x-y)\ud y_d,
\end{aligned}
\end{equation}
where in the last equality we have interchanged $(x',x_d)$ with $(y',y_d)$.
By combining \eqref{tersomma} with \eqref{perio}, we obtain
\begin{equation}\label{finalm}
\begin{aligned}
&\,\per_\sharp^s(E)-\per_\sharp^s(E_\ep)\\
=&\,\int_{\T^{d-1}}\ud x'\int_{\R^{d-1}}\ud y'\int_{(1-\ep)f_E(x')}^{f_E(x')} \ud x_d \int_{f_E(y')}^{+\infty} J^s(x-y)\ud y_d\\
&\,-\int_{\T^{d-1}}\ud x'\int_{\R^{d-1}}\ud y'\int_{(1-\ep)f_E(x')}^{f_E(x')} \ud x_d \int_{-\infty}^{(1-\ep)f_E(y')} J^s(x-y)\ud y_d,
\end{aligned}
\end{equation}
Notice that the righthand side of \eqref{finalm} coincides with the righthand side of \eqref{cozero}, up to replacing $\T^{d-1}$ with $\R^{d-1}$ in the integrals in $y'$ and $K$ with $J^s$; therefore, by arguing verbatim as in the proof of Proposition \ref{propositionfraz}, we get the claim.
\end{proof}
\begin{proposition}\label{Hriesz}
For every $0<\alpha<d-1$, 
the Riesz-type perimeter $\per^\alpha$ in Definition \ref{rieszper} satisfies assumption {\rm (H)}. 
\end{proposition}
\begin{proof}
Let $E\in\Reg^L$. Then, $\rho_E<+\infty$.
Therefore
$$
\per^\alpha(E)=\lim_{R\to +\infty}\per^\alpha_{-R,\rho_E}(E).
$$
We prove that, if \eqref{oscgr} holds true, then for $\ep$ small enough
\begin{equation}\label{perognira}
\per^{\alpha}_{-R,\rho_E}(E)-\per^{\alpha}_{-R,\rho_{E_\ep}}(E_\ep)\ge C(\delta,L)\ep,
\end{equation}
 for every $R>|\rho_E|$ and for some constant $C(\delta,L)$ independent of $R$.
We can assume, without loss of generality, that $f_E>0$, and hence $\rho_E>0$. By the very definition of $\rho_E$ in \eqref{roE}, we get $\rho_{E_\ep}=(1-\ep)\rho_E$. Hence,
\begin{equation*}
\begin{aligned}
&\,\per_{-R,\rho_E}^\alpha(E)-\per_{-R,\rho_{E_\ep}}^\alpha(E_\ep)\\
=& \,  \iint_{\S^d_{-R,\rho_E} \times \S^d_{-R, \rho_E}} \frac{\ud x \ud y}{\|x-y\|_\sharp^{d-\alpha}}  -  \iint_{\S^d_{-R,(1-\ep)\rho_{E}} \times \S^d_{-R, (1-\ep)\rho_{E}}} \frac{\ud x \ud y}{\|x-y\|_\sharp^{d-\alpha}} \\
&\,   - \bigg(\iint_{(E \cap \S^d_{-R,\rho_E}) \times ( E \cap \S^d_{-R,\rho_E})} \frac{\ud x \ud y}{\|x-y\|_\sharp^{d-\alpha}}-\iint_{(E \cap \S^d_{-R,(1-\ep)\rho_{E}}) \times ( E \cap \S^d_{-R,(1-\ep)\rho_{E}})} \frac{\ud x \ud y}{\|x-y\|_\sharp^{d-\alpha}}\bigg).
\end{aligned}
\end{equation*}
By straightforward computations, we obtain
\begin{equation*}
\begin{aligned}
&\,\per_{-R,\rho_E}^\alpha(E)-\per_{-R,\rho_{E_\ep}}^\alpha(E_\ep)\\
=& \, \int_{\T^{d-1}}\ud x' \int_{\T^{d-1}}\ud y' \bigg( \int_{-R}^{\rho_E} \ud x_d \int_{-R}^{\rho_E} \frac{\ud y_d}{\|x-y\|_\sharp^{d-\alpha}}\\
&\,\phantom{\int_{\T^{d-1}}\ud x' \int_{\T^{d-1}}\ud y'}\qquad- \int_{-R}^{(1-\ep)\rho_E} \ud x_d \int_{-R}^{(1-\ep)\rho_E} \frac{\ud y_d}{\|x-y\|_\sharp^{d-\alpha}}\bigg) \\
&\, -\int_{\T^{d-1}}\ud x' \int_{\T^{d-1}}\ud y' 
\bigg( \int_{-R}^{f_E(x')} \ud x_d \int_{-R}^{f_E(y')} \frac{\ud y_d}{\|x-y\|_\sharp^{d-\alpha}}\\
&\,\phantom{\int_{\T^{d-1}}\ud x' \int_{\T^{d-1}}\ud y'}\qquad- \int_{-R}^{(1-\ep)f_E(x')} \ud x_d \int_{-R}^{(1-\ep)f_E(y')} \frac{\ud y_d}{\|x-y\|_\sharp^{d-\alpha}}\bigg) \\
=& \, 2\int_{\T^{d-1}}\ud x' \int_{\T^{d-1}}\ud y' \bigg( \int_{(1-\ep)\rho_E}^{\rho_E} \ud x_d  \int_{-R}^{(1-\ep)\rho_E}\frac{\ud y_d}{\|x-y\|_\sharp^{d-\alpha}}\\
&\,\phantom{\int_{\T^{d-1}}\ud x' \int_{\T^{d-1}}\ud y'}\qquad -  \int_{(1-\ep)f_E(x')}^{f_E(x')} \ud x_d \int_{-R}^{(1-\ep)f_E(y')} \frac{\ud y_d}{\|x-y\|_\sharp^{d-\alpha}}\bigg)\\
&\,-\int_{\T^{d-1}}\ud x' \int_{\T^{d-1}}\ud y' \bigg(
 \int_{(1-\ep)f_E(x')}^{f_E(x')} \ud x_d \int_{(1-\ep)f_E(y')}^{f_E(y')} \frac{\ud y_d}{\|x-y\|_\sharp^{d-\alpha}}\\
 &\,\phantom{\int_{\T^{d-1}}\ud x' \int_{\T^{d-1}}\ud y'}\qquad
 -
 \int_{(1-\ep)\rho_E}^{\rho_E} \ud x_d  \int_{(1-\ep)\rho_E}^{\rho_E}\frac{\ud y_d}{\|x-y\|_\sharp^{d-\alpha}} \bigg).
\end{aligned}
\end{equation*}
For every $0<\ep<1$, we set $\D(\ep):=\per_{-R,\rho_E}^\alpha(E)-\per_{-R,\rho_{E_\ep}}^\alpha(E_\ep)$. Then, to conclude the proof it is enough to show that
\begin{equation}\label{stima}
\D'(0)\ge C(\delta,L),
\end{equation}
for some positive constant $C(\delta,L)$. By the very definition of $\D$, setting $J^\alpha_\sharp(\cdot):=  \|\cdot\|_\sharp^{-d+\alpha}$, we have
\begin{equation*}
\begin{aligned}
\D'(0)=& \, 2\int_{\T^{d-1}}\ud x' \int_{\T^{d-1}}\ud y'  \rho_E  \int_{-R}^{\rho_E} J_{\sharp}^\alpha (x'-y',\rho_E-y_d) \ud y_d \\
&\,-2\int_{\T^{d-1}}\ud x' \int_{\T^{d-1}}\ud y'  f_E(x')  \int_{-R}^{f_E(y')} J_{\sharp}^\alpha (x'-y',f_E(x')-y_d) \ud y_d,
\end{aligned}
\end{equation*}
so that
\begin{equation}\label{derivata}
\begin{aligned}
\frac{\D'(0)}{2}=& \,   \int_{\T^{d-1}}\ud x' \int_{\T^{d-1}}\ud y' \rho_E   \int_{-R}^{\rho_E}  J^\alpha_{\sharp} (x'-y',\rho_E -t) \ud t \\
&\,-\int_{\T^{d-1}}\ud x' \int_{\T^{d-1}}\ud y' f_E(x')   \int_{-R}^{f_E(x')}  J^\alpha_{\sharp} (x'-y',f_E(x') -t) \ud t\\
&\,-\int_{\T^{d-1}}\ud x' \int_{\T^{d-1}}\ud y'f_E(x')\bigg( \int_{-R}^{f_E(y')} J_{\sharp}^\alpha (x'-y',f_E(x')-t) \ud t\\
&\phantom{\int_{\T^{d-1}}\ud x' \int_{\T^{d-1}}\ud y'f_E(x')}- \int_{-R}^{f_E(x')} J_{\sharp}^\alpha (x'-y',f_E(x')-t) \ud t\bigg).
\end{aligned}
\end{equation}
Now, recalling the definitions of  $A^>$ and $A^<$ in \eqref{insiemi}, we have
\begin{equation}\label{terneg}
\begin{aligned}
&\,\int_{\T^{d-1}}\ud x' \int_{\T^{d-1}}\ud y'f_E(x')\bigg( \int_{-R}^{f_E(y')} J_{\sharp}^\alpha (x'-y',f_E(x')-t) \ud t\\
&\phantom{\int_{\T^{d-1}}\ud x' \int_{\T^{d-1}}\ud y'f_E(x')}- \int_{-R}^{f_E(x')} J_{\sharp}^\alpha (x'-y',f_E(x')-t) \ud t\bigg)\\
=& \,\iint_{A^<}\ud x' \ud y'  f_E(x')  \int_{f_E(x')}^{f_E(y')}  J_{\sharp}^\alpha(x'-y', t-f_E(x')) \ud t \\
& \,-\iint_{A^>}\ud x' \ud y'  f_E(x')  \int_{f_E(y')}^{f_E(x')}  J_\sharp^\alpha(x'-y', f_E(x')-t) \ud t\\
=&\,\iint_{A^<}\ud x' \ud y'  f_E(x')  \int_{f_E(x')}^{f_E(y')}  J_{\sharp}^\alpha(x'-y', t-f_E(x')) \ud t \\
& \,-\iint_{A^<}\ud x' \ud y'  f_E(y')  \int_{f_E(x')}^{f_E(y')}  J_\sharp^\alpha(x'-y', f_E(y')-t) \ud t\\
=&\,\iint_{A^<}\ud x'\ud y'(f_E(x')-f_E(y'))\int_{0}^{f_E(y')-f_E(x')}J_\sharp^\alpha(x'-y', \eta) \ud \eta,
%
\end{aligned}
\end{equation}
where we have used that for every $(x',y')\in A^{<}$
\begin{equation*}
\begin{aligned}
\int_{f_E(x')}^{f_E(y')}  J_{\sharp}^\alpha(x'-y', t-f_E(x')) \ud t=&\,\int_{0}^{f_E(y')-f_E(x')} J_{\sharp}^\alpha(x'-y', \eta) \ud\eta\\
\int_{f_E(x')}^{f_E(y')}  J_{\sharp}^\alpha(x'-y', f_E(y')-t) \ud t=&\,\int_{0}^{f_E(y')-f_E(x')} J_{\sharp}^\alpha(x'-y', \eta) \ud\eta,
\end{aligned}
\end{equation*} 
in view of the change of variable $\eta=t-f_E(x')$  in the first integral and $\eta=f_E(y')-t$ in the second one.
Then, by arguing as in the end of the proof of Proposition \ref{propositionfraz}, if \eqref{oscgr} holds true, then the set $A_\delta^{<}$ defined in \eqref{insiemedelta} satisfies \eqref{mispo}, so that, using \eqref{terneg}, we deduce
\begin{equation}\label{terneg1}
\begin{aligned}
&-\,\int_{\T^{d-1}}\ud x' \int_{\T^{d-1}}\ud y'f_E(x')\bigg( \int_{-R}^{f_E(y')} J_{\sharp}^\alpha (x'-y',f_E(x')-t) \ud t\\
&\phantom{\int_{\T^{d-1}}\ud x' \int_{\T^{d-1}}\ud y'f_E(x')}- \int_{-R}^{f_E(x')} J_{\sharp}^\alpha (x'-y',f_E(x')-t) \ud t\bigg)\\
\ge &\,C(\delta,L),
\end{aligned}
\end{equation}
for some $C(\delta,L)>0$. 

Now, for every $x'\in\T^{d-1}$, we define the function $A_{x'}:\R\to [0,+\infty)$ as
\begin{equation*}
A_{x'}(s):=\int_{\T^{d-1}}J^\alpha_{\sharp}(x'-y',s)\ud y'
\end{equation*}
and the function $\Phi_{x'}:[0,+\infty)\to [0,+\infty)$ as
\begin{equation}\label{phi}
\Phi_{x'}(u):= u \int_{-R}^u A_{x'}(u-t) \ud t=u\int_{0}^{R+u}A_{x'}(\xi)\ud \xi,
\end{equation}
where, the second equality follows by the change of variable $\xi=u-t$.
Notice that 
$$
\Phi'_{x'}(u)=\int_{0}^{R+u}A_{x'}(\xi)\ud \xi+uA_{x'}(R+u)>0\qquad\textrm{for every }u> 0.
$$
As a consequence, using that $\rho_E=\max_{\T^{d-1}}f_E$, we get
\begin{equation}\label{terpos1}
\begin{aligned}
&\, \int_{\T^{d-1}}\ud x' \int_{\T^{d-1}}\ud y' \rho_E   \int_{-R}^{\rho_E}  J^\alpha_{\sharp} (x'-y',\rho_E -t) \ud t \\
&\,-\int_{\T^{d-1}}\ud x' \int_{\T^{d-1}}\ud y' f_E(x')   \int_{-R}^{f_E(x')}  J^\alpha_{\sharp} (x'-y',f_E(x') -t) \ud t\\
=&\,\int_{\T^{d-1}}\ud x' \Big(\Phi_{x'}(\rho_E)-\Phi_{x'}(f_E(x'))\Big)> 0.
\end{aligned}
\end{equation}
By \eqref{derivata}, \eqref{terneg1} and \eqref{terpos1}, we get \eqref{stima}.
\end{proof}
\begin{proposition}
The $0$-fractional perimeter in Definition \ref{per0} satisfies assumption {\rm (H)}.
\end{proposition}
\begin{proof}
Let $E\in\Reg^L$. Then, $\rho_E<+\infty$. 
Assume that \eqref{oscgr} holds true, then by arguing verbatim as in the proof of Proposition \ref{Hriesz}, we get that
there exists a constant $C(\delta,L)>0$ such that
\begin{equation}\label{perognira0}
\Q^{0}_{-R,\rho_E}(E)-\Q^{0}_{-R,\rho_{E_\ep}}(E_\ep)\ge C(\delta,L)\ep,
\end{equation}
 for every $R>|\rho_E|$ and for $\ep$ small enough.
Moreover, by applying Proposition \ref{propositionfraz} with $K(\cdot)=\frac{\chi_{B_1}(\cdot)}{\|\cdot\|_{\sharp}^{d}}$, we have that for $\ep$ small enough
\begin{equation}\label{fraz0}
\I^0(E,E^c)-\I^0(E_\ep,E_\ep^c)\ge C(\delta,L)\ep.
\end{equation}
By \eqref{perognira0} and \eqref{fraz0} we deduce the claim.
\end{proof}
\begin{proposition}
For every $\rho>0$, the Minkowski pre-content defined in \eqref{mink2} satisfies assumption {\rm (H)}.
\end{proposition}
\begin{proof}
Recalling \eqref{ingr}, for every $E\in\Reg^L$ and for every $\rho>0$ we set
 \begin{equation}\label{cone}
 \begin{aligned}
 \fat^+_\rho (\partial E):=&\,\{(x',x_d)\in \fat_\rho (\partial E)\,:\,x_d>f_E(x')\}\\
 \fat^-_\rho (\partial E):=&\,\{(x',x_d)\in \fat_\rho (\partial E)\,:\,x_d<f_E(x')\}.
 \end{aligned}
 \end{equation}
By construction, there exist two functions $F_{E,\rho}^+$ and $F_{E,\rho}^-$
such that
 \begin{equation}\label{cone2}
\partial  \fat^\pm_\rho (\partial E)\setminus \partial E=\{(x',F_{E,\rho}^\pm(x'))\,:\,x'\in\T^{d-1}\}.
\end{equation}
Moreover, we notice that for $E\in\Reg^L$
\begin{equation}\label{minko}
\mathcal{E}_\rho(E)=\frac{1}{2\rho}\int_{\T^{d-1}} (F_{E,\rho}^+(x')-F_{E,\rho}^-(x'))\ud x'.
\end{equation}
Let $E\in\Reg^L$ be fixed.
For every $x'\in\T^{d-1}$ let $\xi_{\ep,\rho}^+[x']\in \T^{d-1}$ be such that
$$
|(x',F_{E_\ep,\rho}^+(x'))-(\xi_{\ep,\rho}^+[x'],f_{E_\ep}(\xi_{\ep,\rho}^+[x']))|=\rho,
$$
and let $\vartheta_{E_\ep,\rho}(x')$ be the (smallest) angle formed by the vertical line passing through the point $(x',F_{E_\ep,\rho}^+(x'))$ with the segment 
having extreme points at $(x',F_{E_\ep,\rho}^+(x'))$ and $(\xi_{\ep,\rho}^+[x'],f_{E_\ep}(\xi_{\ep,\rho}^+[x']))$.
By geometric observations, 
\begin{equation}\label{oss1}
(1-\ep)(f_E(\xi_{\ep,\rho}^+[x'])-f_E(x'))= f_{E_\ep}(\xi_{\ep,\rho}^+[x'])-f_{E_\ep}(x')\ge \rho\big(1-\cos(\vartheta_{E_\ep,\rho}(x'))\big).
\end{equation}
By construction, the point $P_\ep:=(x',F_{\ep,\rho}^+(x')+\ep  f_E(\xi_{\ep,\rho}^+[x']))$ has distance equal to $\rho$ from $(\xi_{\ep,\rho}^+[x'], f_E(\xi_{\ep,\rho}^+[x']))$ and hence $\mathrm{dist}(P_\ep, E)\le \rho$.
Therefore, for $\ep$ small enough, we deduce
 \begin{equation}\label{oss2}
 F_{E,\rho}^+(x')\ge F_{E_\ep,\rho}^+(x')+\ep f_E(\xi_{\ep,\rho}^+[x'])\ge F_{E_\ep,\rho}^+(x')+\ep f_E(x')+\ep\frac{\rho}{2}\big(1-\cos (\vartheta_{E_\ep,\rho}(x'))\big),
 \end{equation}
 where the last inequality is a consequence of \eqref{oss1}.
 By \eqref{oss2}, we thus deduce that
 \begin{equation}\label{oss3}
 F_{E,\rho}^+(x')-f_E(x')\ge F_{E_\ep,\rho}^+(x')-(1-\ep)f_E(x')+\ep \frac{\rho}{2}\big(1-\cos (\vartheta_{E_\ep,\rho}(x'))\big).
 \end{equation}
 Now, by construction, $F_{E_\ep,\rho}^+$ is $L$-Lipschitz continuous;
moreover, if $F_{E_\ep,\rho}^+$ is differentiable at $x'$, then 
 \begin{equation}\label{oss4}
 |\nabla F_{E_\ep,\rho}^+(x')|\le \sin(\vartheta_{E_\ep,\rho}(x')).
 \end{equation}
By \eqref{oss3} and \eqref{oss4}, for a.e. $x'\in\T^{d-1}$
\begin{equation}\label{oss5}
F_{E,\rho}^+(x')-f_E(x')\ge F_{E_\ep,\rho}^+(x')-(1-\ep)f_E(x')+\ep \frac{\rho}{4} |\nabla F_{E_\ep,\rho}^+(x')|^2.
\end{equation}
We claim that if \eqref{oscgr} holds, then for $\ep$ small enough
\begin{equation}\label{oss6}
\|\nabla F^+_{E_\ep,\rho}\|^2_{L^2(\T^{d-1})}\ge C(\delta),
\end{equation}
for some positive constant $C(\delta)$.
Indeed, assume by contradiction that there exists a vanishing sequence $\{\ep^n\}_{n\in\N}$
and
a sequence $\{E^n\}_{n\in\N}\in\Reg^L$ satisfying \eqref{oscgr} and such that 
$$
\|\nabla F^+_{E^n_{\ep^n},\rho}\|_{L^2(\T^{d-1})}\to 0\qquad\textrm{ as }n\to +\infty.
$$
Then, since the functions $F^+_{E^n_{\ep^n},\rho}$ are $L$-Lipschitz continuous, we get that $F_{E_{\ep^n}^n,\rho}^+$ uniformly converge to some constant $c$ as $n\to +\infty$; as a consequence, $F_{E^n,\rho}^+\to c$ uniformly in $\T^{d-1}$ thus contradicting \eqref{oscgr}. By \eqref{oss5} and \eqref{oss6} we thus obtain
\begin{equation}\label{fi1}
\int_{\T^{d-1}} (F_{E,\rho}^+(x')-f_E(x'))\ud x'-\int_{\T^{d-1}} (F_{E_\ep,\rho}^+(x')-f_{E_\ep}(x'))\ud x'\ge\ep\frac{\rho}{4}C(\delta);
\end{equation}
analogously,
\begin{equation}\label{fi2}
\int_{\T^{d-1}} (f_E(x')-F^-_{E,\rho}(x'))\ud x'-\int_{\T^{d-1}} (f_{E_\ep}(x')-F_{E_\ep,\rho}^-(x'))\ud x'\ge\ep\frac{\rho}{4}C(\delta).
\end{equation}
In view of \eqref{minko}, by \eqref{fi1} and \eqref{fi2}, we obtain \eqref{oranu}.
\end{proof}

\end{document}